\documentclass[12pt,reqno]{amsart}
\usepackage{amssymb}
\usepackage{mathtools}
\usepackage{txfonts}
\usepackage{eucal}
\usepackage{extarrows}
\usepackage{amscd}
\usepackage[all]{xy}
\usepackage{amsmath}
\usepackage{tikz}
\usepackage{graphicx}
\usepackage{float}
\usepackage{xspace}

\usepackage{multirow}
\usepackage{makecell}
\usepackage{booktabs}

\usepackage[a4paper,body={16.3cm,22.8cm},centering]{geometry}
\usepackage[colorlinks,final,backref=page,hyperindex,pdftex]{hyperref}

\newcommand{\delete}[1]{}

\newcommand{\mcite}[1]{\cite{#1}}  
\newcommand{\mbibitem}[1]{\bibitem{#1}} 

\delete{

\newcommand{\mcite}[1]{\cite{#1}{\small{\tt{{\ }(#1)}}}}  
\newcommand{\mbibitem}[1]{\bibitem[\bf #1]{#1}} 
}
\newtheorem{thm}{Theorem}[section]
\newtheorem{prop}[thm]{Proposition}
\newtheorem{lem}[thm]{Lemma}

\theoremstyle{definition}
\newtheorem{defn}[thm]{Definition}
\newtheorem{rem}[thm]{Remark}
\newtheorem{exam}[thm]{Example}
\newtheorem{prop-def}{Proposition-Definition}[section]

\newcommand{\rr}{\big(R,\cdot,(P_{\omega})_{\omega\in S},\theta\big)}
\newcommand{\ee}{\big(E,\cdot_{E},(P_{\omega,\,E})_{\omega\in S},\theta_{E}\big)}
\newcommand{\eep}{\big(E,\cdot_{E}^{\prime},(P_{\omega,\,E}^{\prime})_{\omega\in S},\theta_{E}^{\prime}\big)}

\newcommand{\crr}{R}
\newcommand{\cee}{E}
\newcommand{\rnv}{R\natural V}
\newcommand{\rnvp}{R\natural^{\prime} V}

\newcommand{\cvv}{V}
\newcommand{\cxx}{V}

\newcommand{\nc}{\newcommand}
\nc{\tred}[1]{\textcolor{red}{#1}} \nc{\tgreen}[1]{\textcolor{green}{#1}}
\nc{\tblue}[1]{\textcolor{blue}{#1}} \nc{\tpurple}[1]{\textcolor{purple}{#1}}
\nc{\bfk}{{\bf k}}
\nc{\yuan}[1]{\tred{\underline{Yuan:}#1 }}
\nc{\wen}[1]{\tblue{\underline{Wen:}#1 }}

\begin{document}

\title[Extending Structures for Rota-Baxter family Hom-associative Algebras]{Extending Structures for Rota-Baxter family Hom-associative Algebras}

\author{Junwen Wang}
\address{School of Mathematics and Statistics,  Henan University,  Henan,  Kaifeng 475004,  P.\, R. China}
\email{3111435107@qq.com}


\author{Yuanyuan Zhang$^{*}$}\thanks{*Corresponding author}
\address{School of Mathematics and Statistics,  Henan University,  Henan,  Kaifeng 475004,  P.\, R. China}
\email{zhangyy17@henu.edu.cn}
\author{Yanjun Chu}
\address{School of Mathematics and Statistics,  Henan University,  Henan,  Kaifeng 475004,  P.\, R. China}
\email{chuyj@henu.edu.cn}

\date{\today}

\begin{abstract}%
In this paper, we first define extending datums and unified products of Rota-Baxter family Hom-associative algebras, and theoretically solve the extending structure problem. Moreover, we consider flag datums as an application, and give an example of the extending structure problem. Second, we introduce matched pairs of Rota-Baxter family Hom-associative algebras, and theoretically solve the factorization problem. Finally, we define deformation maps on a Rota-Baxter family Hom extending structure, and theoretically solve the classifying complements problem.
\end{abstract}

\makeatletter
\@namedef{subjclassname@2020}{\textup{2020} Mathematics Subject Classification}
\makeatother


\maketitle

\tableofcontents

\setcounter{section}{0}

\allowdisplaybreaks



\section{Introduction}
\subsection{Rota-Baxter family algebras}
Let $S$ be a semigroup. A Rota-Baxter family algebra of weight $\lambda$~\cite{Guo09} is a pair $(R,(P_\omega)_{\omega\in S })$ consisting of an associative algebra $R$ over some field $\bfk$ together with a collection $(P_\omega)_{\omega\in S }:R\to R$ of linear endomorphisms indexed by a semigroup $ S $ such that the Rota-Baxter family relation
\begin{equation*}
P_{\alpha}(a)P_{\beta}(b)=P_{\alpha\beta}\bigl( P_{\alpha}(a)b  + a P_{\beta}(b) + \lambda ab \bigr)
\end{equation*}
 holds for any $a, b \in R$ and $\alpha,\, \beta \in  S $.
The first family algebra structure appeared in the literature of Ebrahimi-Fard et al. on Lie-theoretic aspects of renormalization~\cite[Proposition~9.1]{EGP} (see also~\cite{DK}). The example in \cite{EGP} is given by the momentum renormalization scheme: here $ S $ is the additive semigroup of non-negative integers, and the operator $P_\omega$ associates to a Feynman diagram integral its Taylor expansion of order $\omega$ at vanishing exterior momenta.
 Other families of algebraic structures appeared recently: dendriform and tridendriform family algebras \mcite{A2020,Foi20,ZG, ZGM}, pre-Lie family algebras \cite{MZ}. In \cite{Das}, the author introduced the notion of a relative Rota-Baxter family algebra and found various relations with dendriform family algebras. In~\cite{Das1}, the author define the cohomology of a given Rota-Baxter family algebra.

\smallskip

%

\subsection{Hom-algebras}

    Hom-type algebras are the corresponding algebraic identities twisted by a linear space homomorphism. Recently, Hom-type algebras have been studied by many authors. The notion of Hom-Lie algebras
was first introduced by Hartwig, Larsson and Silvestrov \cite{HLS}. Hom-Lie algebras appeared
in examples of q-deformations of the Witt and Virasoro algebras. Other type of algebras
(e.g. associative, Leibniz, Poisson, pre-Lie,...) twisted by homomorphisms have also been
studied. See \cite{MS, MS2} (and references therein) for more details.
\smallskip

A Hom-associative algebra is a multiplication on a vector space where the structure is twisted by a linear space homomorphism, which is a generalization of associative algebras ~\cite{MS}.
\begin{defn}
Let $V$ be a \bfk-vector space. A Hom-associative algebra over $V$ is a triple $(V, \mu, \theta)$ where $\mu: V \times V \rightarrow V$ is a bilinear map and $\theta: V \rightarrow V$ is a linear map, satisfying
\[\quad \mu(\theta(x), \mu(y, z))=\mu(\mu(x, y), \theta(z)),\,\text{for}\, x, y,z\in V.\]
\end{defn}
Makhlouf~\cite{M12} studied Rota-Baxter Hom-algebras, Hom-dendriform algebras and Hom-preLie algebras and generalized the canonical relation to the ``Hom" version.
Attan etc ~\cite{AGK} studied Rota-Baxter family Hom-associative algebras, Hom-(tri)dendriform family algebras and Hom-preLie family algebras, and generalized the canonical relation to the ``Hom and family" context. Our main objective in this
paper is the notion of Rota-Baxter family Hom-associative algebra.

%
%
\smallskip
\subsection{Extending structures}
The extending structure problem was first studied in group theory by Agore and Militaru~\cite{AM14}, combining the bicrossed product (matched pair)~\cite{T81} and the crossed product~\cite{AM08}, they obtained a more general product (called unified product),
and then they theoretically solved the extending structure problem. Later, they also studied Lie algebras~\cite{AM14L}, Leibniz algebras~\cite{AM13}, Hopf algbras\cite{AM13B}, poisson algebras~\cite{AM15} and associative algebras~\cite{AM16}. . In recent years, many researchers studied the problem for various algebraic structures. Hong studied extending structures for left-symmetric algebras~\cite{H19b}, conformal algebras~\cite{H19a,H17} and Lie bialgebras~\cite{H21}, Zhang studied for 3-Lie algebras~\cite{Z22b} and braided Lie bialgebras~\cite{Z22a}, Zhao et.al studied for Lie conformal superalgebras~\cite{ZCY}, Peng et.al studied for Rota-Baxter Lie algebras~\cite{PZ}, Hou studied for perm algebras~\cite{HB23}, and so on. The scholars
who care about this problem often define extending structures and unified products
of different algebraic structures. In this paper, we will define extending structures and unified products of Rota-Baxter family Hom-associative algebras as a useful tool to solve the extending structure problem for Rota-Baxter family Hom-associative algebras.

\begin{defn}
If $\crr$ is a subalgebra of Rota-Baxter family Hom-associative algebra $E$, then $E$ is also called an {\bf extension} of $\crr$, and is denoted by $\crr \subset E$. In this case, we also say $\crr\subset E$ is an extension of Rota-Baxter family Hom-associative algebras. A subspace $V$ of $E$ is called a {\bf space complement} of $\crr$ in $E$ if $E=\crr+V$, and $\crr\cap V={0}$.
\end{defn}

Our first aim of this paper is to study the extending structure problem for Rota-Baxter family Hom-associative algebras.

{\bf The Extending Structure (ES) Problem}: Let $\crr$ be a Rota-Baxter family Hom-associative algebra, $E$ a vector space containing $\crr$ as a subspace. Describe and classify all extensions $\ee$ of $\crr$.
The ES problem can be rephrased with the language of ``subalgebra": Let $\crr$ be a Rota-Baxter family Hom-associative algebra, $E$ a vector space containing $\crr$ as a subspace. Describe and classify all Rota-Baxter family Hom-associative algebra structures that can be defined on $E$ containing $\crr$ as a subalgebra.
\smallskip

Our second aim of this paper is to study the factorization problem for Rota-Baxter family Hom-associative algebras, which is the ES problem with an additional assumption ``the space complement $V$ of $\crr$ in $E$ to be also a subalgebra of $E$".

{\bf The Factorization Problem}: Let $\crr$ and $\cvv$ be two Rota-Baxter family Hom-associative algebras. Describe and classify all Rota-Baxter family Hom-associative algebra structures that can be defined on $E$ such that $E$ {\bf factorizes through} $\crr$ and $V$, i.e., $E$ contains $\crr$ and $V$ as subalgebras such that $E=R+V$ and $R\cap V=\{0\}$.


\begin{defn}
Let $\crr \subset \cee$ be an extension of Rota-Baxter family Hom-associative algebras. A subalgebra $B$ of $\cee$ is called a {\bf Rota-Baxter family Hom complement} of $\crr$ in $\cee$ if $E=R+B$ and $R\cap B=\{0\}$. We know that $V$ is a Rota-Baxter family Hom complement of $\crr$ in $E$ if and only if $E$ factorizes through $\crr$ and $V$.
\end{defn}

Our third aim of this paper is to study the classifying complements problem for Rota-Baxter family Hom-associative algebras.
Let $\crr,V,E$ be three Rota-Baxter family Hom-associative algebras such that $E=\crr+V$ and $R\cap V=\{0\}$. The factorization problem is to find all $E$ by fixing $\crr$ and $V$, while the classifying complements problem is to find all $V$ by fixing $\crr$ and $E$. Generally speaking, the classifying complements problem is an inverse problem of the factorization problem.

{\bf The Classifying Complements Problem(CCP)}: Let $\crr \subset E$ be an extension of Rota-Baxter family Hom-associative algebras.
%
Describe and classify all Rota-Baxter family Hom complements of $\crr$ in $E$, and compute the cardinal of the isomorphism classes of all Rota-Baxter family Hom complements of $\crr$ in $E$, which is called the {\bf index} of $\crr$ in $E$ and will be denoted by $[E:\crr]$.

The CCP problem was first studied in group theory~\cite{AM15GC}, and an additional condition ``if a group complement exists" is needed. The authors define deformation maps on a group complement (i.e. a matched pair) and theoretically solve the problem. They also studied the CCP problems of associative algebras, Lie algebras and Hopf algebras by defining deformation maps on the given algebraic complement (i.e. a matched pair)~\cite{A14,AM13HC,AM14LC}. In this way, Hong~\cite{H19b} and Hou~\cite{HB23} studied the CCP Problems of left-symmetric algebras and perm algebras respectively. However, we define deformation maps on a dendriform extending structure (not a matched pair) and solve the CCP problem more practically in~\cite{ZW24}.
 In this paper, we define deformation maps on a Rota-Baxter family Hom extending structure (more general case), still not necessary a matched pair.

\smallskip
\subsection{Outline of the paper}
The paper is organized as follows. In Section 2, we recall some basic concepts of Rota-Baxter family Hom-associative algebras. In Section 3, we first define extending datums and unified products of Rota-Baxter family Hom-associative algebras. Second, we establish a bijective map between extensions and Rota-Baxter family Hom extending structures, which induces a bijection between the equivalent (cohomologous) classes by defining an equivalent (cohomologous) relation. This gives a theoretically answer to the extending structure problem. Third, we consider the flag datums as a special case of Rota-Baxter family Hom extending structures, and give an example of the extending structure problem. Finally, we introduce matched pairs and bicrossed products of Rota-Baxter family Hom-associative algebras and theoretically solve the factorization problem. In Section 4, we define the deformation map on a Rota-Baxter family Hom extending structure and theoretically solve the classifying complements problem.

\smallskip
{\bf Notation.}
Throughout this paper, let $\bfk$ be a field unless the contrary is specified, which will be
the base ring of all modules, algebras, as well as linear
maps. By an associative (or Hom-associative) algebra we mean a nonunitary (not necessary unital) associative algebra. We also fix $S$ a semigroup whose elements will be denoted by $\omega,\alpha,\beta,\cdots$. A linear (or bilinear) map is called {\bf trivial} if it is 0. We always assume that $\lambda \in \bf k$.

\section{Rota-Baxter family Hom-associative algebras and extension relations}

\subsection{Rota-Baxter family Hom-associative algebras and modules}
In this subsection, we mainly recall some basic concepts of Rota-Baxter family Hom-associative algebras and give some examples.

%
%
%

\begin{defn}~\cite{AGK}
Let $R$ be a vector space, $\cdot:R\times R\rightarrow R$ a bilinear map, $(P_{\omega})_{\omega\in S}:R\rightarrow R$ a family of linear maps and $\theta:R\rightarrow R$ a linear map.
A {\bf Rota-Baxter family Hom-associative algebra} of weight $\lambda$, or simply {\bf Rota-Baxter family Hom-associative algebra} is a 4-tuple $\rr$ if the following identities hold for all $x,y,z\in R$ and $\omega,\alpha,\beta\in S$:
\begin{align*}
\theta(x)\cdot(y\cdot z)=&(x\cdot y)\cdot\theta(z);\\
P_{\alpha}(x)\cdot P_{\beta}(y)=&P_{\alpha\beta}\big(P_{\alpha}(x)\cdot y+x\cdot P_{\beta}(y)+\lambda\,x\cdot y\big);\\
P_{\omega}\big(\theta(x)\big)=&\theta\big( P_{\omega}(x)\big).
\end{align*}
\end{defn}
  The 4-tuple $\rr$ will be denoted simply by $R$ and the multiplication $\cdot$ will be simply denoted by concatenation
   if there is no confusion. 

\smallskip

\begin{rem}~\cite{AGK}
Let $\big(R,\cdot,(P_{\omega})_{\omega\in S}\big)$ be a Rota-Baxter family algebra  . Suppose that $\theta:R\rightarrow R$ is a linear map, satisfying $\theta(x y)=\theta(x)\,\theta(y)$, $P_{\omega}\big(\theta(x)\big)=\theta\big(P_{\omega}(x)\big)$, for all $x,y\in R,\omega\in S$. Then the 4-tuple $\big(R,\cdot_{\theta}:=\theta\circ\cdot,(P_{\omega})_{\omega\in S},\theta\big)$ is a Rota-Baxter family Hom-associative algebra, where the operation ``$\circ$" is the composition of maps.
\end{rem}

\begin{exam}\label{exam:dendriform algebra}Let $S=\{e,\sigma\}$ be a semigroup, where the element $e$ is a unit and $\sigma^{2}=e$. The Rota-Baxter family $(P_{\omega})_{\omega\in S}$ is denoted by $(P_{e},P_{\sigma})$, and the weight $\lambda\neq 0$.
\begin{enumerate}
\item\label{exam:dendriform algebra1} Suppose that $R={\bf k}\{e_{1}\}$, $e_{1}e_{1}=e_{1}$, $\theta=\mathrm{Id}_{R}$, $(P_{e},P_{\sigma})=(0,0)$. By directly computing, the 4-tuple $(R,\cdot,(P_{e},P_{\sigma}),\theta)$ is a Rota-Baxter family Hom-associative algebra.
\item\label{exam:dendriform algebra2} Suppose that $B={\bf k}\{e_{2}\}$, $e_{2}\cdot_{B} e_{2}=e_{2}$, $\theta_{B}(e_{2})=e_{2}$, $P_{e,\,B}(e_{2})=P_{\sigma,\,B}(e_{2})=-\lambda e_{2}$, then $(B,\cdot_{B},(P_{e,\,B},P_{\sigma,\,B}),\theta_{B})$ is a Rota-Baxter family Hom-associative algebra.
\item Suppose that $E={\bf k}\{e_{1},e_{2}\}$ satisfying the following conditions
\begin{align*}
\begin{aligned}
e_{1}\cdot_{E}e_{1}&=e_{1},\\
e_{2}\cdot_{E}e_{1}&=3e_{1},\\
P_{e,\,E}(e_{1})&=0,\\
P_{\sigma,\,E}(e_{1})&=0,\\
\theta_{E}(e_{1})&=e_{1},\\
\end{aligned}\quad
\begin{aligned}
e_{1}\cdot_{E}e_{2}&=-3e_{1}+2e_{2},\\
e_{2}\cdot_{E}e_{2}&=-9e_{1}+6e_{2},\\
P_{e,\,E}(e_{2})&=3\lambda e_{1}-\lambda e_{2},\\
P_{\sigma,\,E}(e_{2})&=3\lambda e_{1}-\lambda e_{2},\\
\theta_{E}(e_{2})&=-3e_{1}+2e_{2}.
\end{aligned}
\end{align*}
Then the 4-tuple $(E,\cdot_{E},(P_{e,\,E},P_{\sigma,\,E}),\theta_{E})$ is a Rota-Baxter family Hom-associative algebra.
\end{enumerate}
\end{exam}

\smallskip

\begin{defn}~\cite{AGK}
 Let $\crr_{1}$ and $R_{2}$ be two Rota-Baxter family Hom-associative algebras. A linear map $\varphi:\crr_{1}\rightarrow \crr_{2}$ is called a {\bf morphism} of Rota-Baxter family Hom-associative algebras if the following identities hold for all $x,y\in \crr_{1}$ and $\omega \in S$:
   \allowdisplaybreaks{
    \begin{align*}
   \varphi(x\cdot_{\crr_{1}} y)=\varphi(x)\cdot_{\crr_{2}} \varphi(y),\quad
  \varphi\big(P_{\omega,\,\crr_{1}}(x)\big)=P_{\omega,\,\crr_{2}}\big(\varphi(x)\big),\quad
    \varphi\big(\theta_{\crr_{1}}(x)\big)=\theta_{\crr_{2}}\big(\varphi(x)\big).
    \end{align*}
    }A morphism $\varphi$ is called an {\bf isomorphism} if it is also a bijective map. Two Rota-Baxter family Hom-associative algebras $\crr_{1}$ and $\crr_{2}$ are called {\bf isomorphic}, and we denote it by $\crr_{1}\cong \crr_{2}$, if there exists an isomorphism of Rota-Baxter family Hom-associative algebras $\varphi:\crr_{1}$ $\rightarrow \crr_{2}$.
\end{defn}

The module of Hom-associative algebras is defined in ~\cite{MS102}, and the module of Rota-Baxter Lie algebras is defined in ~\cite{PZ}, analogously we have the following definition.
\begin{defn} Let $\crr$ be a Rota-Baxter family Hom-associative algebra, $V$ a vector space and $\theta_{V},(P_{\omega,\,V})_{\omega\in S}:V\rightarrow V$ linear maps. Suppose that $\triangleright:R\times V\rightarrow V$ and $\triangleleft:V\times R\rightarrow V$ are both bilinear maps.
\begin{enumerate}
  \item The 4-tuple $\big(V,(P_{\omega,\,V})_{\omega\in S},\theta_{V},\triangleright\,(\mathrm{resp.}\, \triangleleft))$ is called a {\bf left (resp. right) $\crr$ module} if Eqs. ~\eqref{formulas:left module definition1}-\eqref{formulas:left module definition3} (resp. Eqs. ~\eqref{formulas:left module definition3}-\eqref{formulas:right module definition5}) hold for all $a,b\in R$, $x\in V$ and $\alpha,\beta\in S$:
    \begin{align}
   (ab)\triangleright \theta_{V}(x)&=\theta(a)\triangleright(b\triangleright x);\label{formulas:left module definition1}\\
  P_{\alpha}(a)\triangleright P_{\beta,\,V}(x)&=P_{\alpha\beta,\,V}\big(P_{\alpha}(a)\triangleright x+a\triangleright P_{\beta,\,V}(x)+\lambda\,a\triangleright x\big);\label{formulas:left module definition2}\\
  \theta_{V}\big(P_{\alpha,\,V}(x)\big)&=P_{\alpha,\,V}\big(\theta_{V}(x)\big);\label{formulas:left module definition3}\\
      \theta_{V}(x)\triangleleft(ab)&=(x\triangleleft a)\triangleleft \theta(b);\label{formulas:right module definition4}\\
      P_{\alpha,\,V}(x)\triangleleft P_{\beta}(a)&=P_{\alpha\beta,\,V}\big(P_{\alpha,\,V}(x)\triangleleft a+x\triangleleft P_{\beta}(a)+\lambda\,x\triangleleft a\big).\label{formulas:right module definition5}
    \end{align}
    \item The 5-tuple $\big(V,(P_{\omega,\,V})_{\omega\in S},\theta_{V},\triangleright,\triangleleft\big)$ is called a $\crr$ {\bf bimodule}, if the 4-tuple $\big(V,(P_{\omega,\,V})_{\omega\in S},\theta_{V},\triangleright\big)$ is a left $\crr$ module, $\big(V,(P_{\omega,\,V})_{\omega\in S},\theta_{V},\triangleleft\big)$ is a right $\crr$ module and the following identity holds for all $a,b\in R$ and $x\in V$:
    \begin{align}\label{formulas:bimodule definition}
         (a\triangleright x)\triangleleft \theta(b)&=\theta(a)\triangleright(x\triangleleft b).
    \end{align}
  \item Let $\big(V,(P_{\omega,\,V})_{\omega\in S},\theta_{V},\triangleright_{V}\big)$ and $\big(W,(P_{\omega,\,W})_{\omega\in S},\theta_{W},\triangleright_{W}\big)$ be
two left $\crr$ modules. A linear map $\varphi:V\rightarrow W$ is called a {\bf left $\crr$ module morphism} if the following identities hold for all $a\in R$, $x\in V$ and $\omega\in S$:
    \begin{align*}
  \varphi(a\triangleright_{V} x)=a\triangleright_{W}\varphi(x),\quad
  \varphi\big(P_{\omega,\,V}(x)\big)=P_{\omega,\,W}\big(\varphi(x)\big),\quad
    \varphi\big(\theta_{V}(x)\big)=\theta_{W}\big(\varphi(x)\big).
    \end{align*}
The right $\crr$ module morphism can be defined similarly. We call the linear map a $\crr$ bimodule morphism if it is both a left $\crr$ module morphism and a right $\crr$ module morphism.
   \end{enumerate}
\end{defn}

\subsection{The extension relations of Rota-Baxter family Hom-associative algebras}
Let $\crr$ be a Rota-Baxter family Hom-associative algebra, $E$ a vector space containing $R$ as a subspace, and $V$ a space complement of $R$ in $E$. In this case, an element of $E$ will be denoted by the form $a+x$ with $a\in R$ and $x\in V$. We denote by $\mathbf{Exts}(E,R)$ the set of extensions $\ee$ of $\crr$.
\begin{defn}\label{defn:equivalent cohomolous}
Let $\ee$ and $\eep$ be in $\mathbf{Exts}(E,R)$. In the following Diagram ~(\ref{diagram:equivalent cohomologous ee}), the inclusion map $i:R\rightarrow \cee$ and the canonical projection $\pi:\cee\rightarrow V$ are defined by
 $$i(a)=a,\,\pi(a+x)=x,\quad \text{ for }a\in R,\,x\in V,$$
then the two rows are both short exact sequences in the category $\mathbf{Vect}_{\bf k}$ of vector spaces. We say {\bf a linear map $\psi:\ee \rightarrow \eep$ stabilizes $R$ }if the left square is commutative in the category $\mathbf{Vect}_{\bf k}$. Similarly, we say {\bf $\psi$ co-stabilizes $V$ }if the right square is commutative in the category $\mathbf{Vect}_{\bf k}$.
\begin{enumerate}
\item Two extensions are called {\bf equivalent} if there exists an isomorphism of Rota-Baxter family Hom-associative algebras $\psi:\ee \rightarrow \eep$ which stabilizes $R$ in Diagram ~(\ref{diagram:equivalent cohomologous ee}). We denote it by $\ee\equiv \eep$.
\item Two equivalent extensions are called {\bf cohomologous} if the map $\psi$ also co-stabilizes $V$ in Diagram ~(\ref{diagram:equivalent cohomologous ee}). We denote it by $\ee\approx \eep$.
\end{enumerate}
\begin{align}\label{diagram:equivalent cohomologous ee}
\begin{split}
\xymatrix{
  0 \ar[r]^{} & R \ar[d]_{\mathrm {Id}} \ar[r]^{i\qquad\qquad} & \ee \ar[d]_{\psi} \ar[r]^{\qquad\qquad\pi} & V \ar[d]_{\mathrm {Id}} \ar[r]^{} & 0  \\
  0 \ar[r]^{} & R \ar[r]^{i\qquad\qquad} & \eep \ar[r]^{\qquad\qquad\pi} & V \ar[r]^{} & 0   }
\end{split}
\end{align}
\end{defn}

The relations $\equiv$ and $\approx$ defined on the set $\mathbf{Exts}(E,R)$ are both equivalence relations.
Suppose that $$\mathbf{Extd}(E,R):=\mathbf{Exts}(E,R)/\equiv,\qquad \mathbf{Extd^{\prime}}(E,R):=
\mathbf{Exts}(E,R)/\approx,$$
then the set $\mathbf{Extd}(E,R)$ gives a classification of the $ES$ problem and the set $\mathbf{Extd^{\prime}}(E,R)$ is another more strictly classifying object for the ES problem. Since any two cohomologous extensions are obviously equivalent, there exists a canonical projection $\mathbf{Extd^{\prime}}(E,R)\twoheadrightarrow \mathbf{Extd}(E,R)$.

For later use, we recall a basic conclusion of the set theory.
\begin{prop}\label{prop:induce map}
Let $A$ and $B$ be two sets. Let $\xi:A\rightarrow B$ be a bijective map, and $\approx$ an equivalence relation on $A$. Define a relation on $B$ as follows: $b_{1}\approx b_{2}$ if and only if $\xi^{-1}(b_{1})\approx \xi^{-1}(b_{2})$, for all $b_{1},b_{2}\in B$. Then we have the following properties:
\begin{enumerate}
\item\label{prop set item 1} The relation $\approx $ is an equivalence relation on $B$.
\item\label{prop set item 2}  The map $\xi$ induces a bijection $\bar{\xi}:A/\approx\,\,\rightarrow B/\approx,\, \bar{a}\,\mapsto \overline{\xi(a)}$, where $\bar{a}$ and $\overline{\xi(a)}$ are the equivalence classes of $a$ and $\xi(a)$, respectively.
\end{enumerate}
\end{prop}
\begin{proof}
Since the proof of \ref{prop set item 1} is a direct computation, here we only proof \ref{prop set item 2}.

First, we prove that the map $\bar{\xi}$ is well defined. In fact, if $\overline{{a}_{1}}=\overline{{a}_{2}}$ with $a_{1},a_{2}\in A$, we have $a_{1}\approx a_{2}$, i.e., $\xi^{-1}(\xi(a_{1}))\approx\xi^{-1}(\xi(a_{2}))$, hence $\xi(a_{1})\approx\xi(a_{2})$, i.e., $\overline{\xi(a_{1})}=\overline{\xi(a_{2})}$.

Second, we prove that the map $\bar{\xi}$ is surjective. In fact, for all $\bar{b}\in B/\approx$, since $\xi$ is bijective, there exists a unique element $a\in A$ such that $\xi(a)=b$, hence $\bar{\xi}(\bar{a})=\overline{\xi(a)}=\bar{b}$.

Third, we prove that the map $\bar{\xi}$ is injective. In fact, for all $\overline{a_{1}},\overline{a_{2}}\in A/\approx$, if $\bar{\xi}(\overline{{a}_{1}})=\bar{\xi}(\overline{{a}_{2}})$, i.e., $\overline{\xi(a_{1})}=\overline{\xi(a_{2})}$, we have $\xi(a_{1})\approx\xi(a_{2})$, hence $a_{1}\approx a_{2}$, i.e., $\overline{{a}_{1}}=\overline{{a}_{2}}$.
\end{proof}

\section{Extending structure for Rota-Baxter family Hom-associative algebras}
In this section, we define extending datums and unified products of Rota-Baxter family Hom-associative algebras, and theoretically solve the ES problem. Moreover, we define flag datums and give an example of ES problem for Rota-Baxter family Hom-associative algebras. Finally, we introduce matched pairs of Rota-Baxter family Hom-associative algebras, and theoretically solve the factorization problem.

\subsection{Unified products of Rota-Baxter family Hom-associative algebras}

\begin{defn}
Let $\crr$ be a Rota-Baxter family Hom-associative algebra and $V$ a vector space.
Suppose that
\begin {align*}
\begin{aligned}
\triangleright&:\,R\times V\rightarrow V,\\
\rightharpoonup&:\,V\times R\rightarrow R,\\
\end{aligned}\quad
\begin{aligned}
\triangleleft&:\,V\times R\rightarrow V,\\
\leftharpoonup&:\,R\times V\rightarrow R,\\
\end{aligned}\quad
\begin{aligned}
f&:\,V\times V\rightarrow R,\\
\cdot_{V}&:\,V\times V\rightarrow V
\end{aligned}
\end{align*}
are bilinear maps, and
\begin {align*}
(Q_{\omega})_{\omega\in S}:\, V\rightarrow R,\quad
(P_{\omega,\,V})_{\omega\in S}:\,V\rightarrow V,\quad
\eta:\,V\rightarrow R, \quad
\theta_{V}:\,V\rightarrow V
\end{align*}
are linear maps. The system $\Omega(R,V)=\big(\triangleright,\triangleleft,\rightharpoonup,\leftharpoonup,f,\cdot_{V},(Q_{\omega})_{\omega\in S},(P_{\omega,\,V})_{\omega\in S},\eta,\theta_{V}\big)$ is called an {\bf extending datum} of $\crr$ through $V$.
\end{defn}
\begin{defn}
Let $\Omega(R,V)=(\triangleright,\triangleleft,\rightharpoonup,\leftharpoonup,f,\cdot_{V},(Q_{\omega})_{\omega\in S},(P_{\omega,\,V})_{\omega\in S},\eta,\theta_{V})$ be an extending datum.
For all $x,y\in V$, $a,b\in R$ and $\omega \in S$, define
\begin{align}\label{formulas:unified product}
\begin{split}
(a,x)\,\bar{\cdot}\,(b,y)&:=(ab+a\leftharpoonup y+x\rightharpoonup b+f(x,y),\,a\triangleright y+x\triangleleft b+x\cdot_{V} y);\\
\bar{P}_{\omega}(a,x)&:=\big(P_{\omega}(a)+Q_{\omega}(x),\,P_{\omega,\,V}(x)\big);\\
\bar{\theta}(a,x)&:=\big(\theta(a)+\eta(x),\,\theta_{V}(x)\big).
\end{split}
\end{align}
We denote the 4-tuple $\big(R\times V,\,\bar{\cdot}\,,(\bar{P}_{\omega})_{\omega\in S},\bar{\theta}\big)$ by $R\natural_{\Omega(R,V)} V=\rnv$ for simplicity. The object $\rnv$ is called a {\bf unified product} of $\crr$ and $\Omega(R,V)$, if $\rnv$ is a Rota-Baxter family Hom-associative algebra. In this case, the extending datum $\Omega(R,V)$ is called a {\bf Rota-Baxter family Hom extending structure} of $\crr$ through $V$. The maps $\triangleright,\triangleleft,\leftharpoonup$ and $\rightharpoonup$ are called {\bf actions} of $\Omega(R,V)$, $f$ is called the {\bf cocycle} of $\Omega(R,V)$.
\end{defn}
We denote by $\mathcal{O}(R,V)$ the set of all Rota-Baxter family Hom extending structures of $\crr$ through $V$.

\begin{thm}\label{thm:datum and unified product}
Let $\crr$ be a Rota-Baxter family Hom-associative algebra, $V$ a vector space and $\Omega(R,V)$ an extending datum of $\crr$ through $V$. Then the following statements are equivalent:
\begin{enumerate}
\item The object $\rnv$ is a unified product;
\item The following conditions hold for all $a,b\in R$, $x,y,z\in V$ and $\alpha,\beta\in S$:
  \begin{align*}
        (R1)\,\,\,\,&\text{$\big(V,(P_{\omega,\,V})_{\omega\in S},\theta_{V},\triangleright,\triangleleft\big)$ is a $\crr$ bimodule, i.e., satisfying {\rm Eqs.} ~\eqref{formulas:left module definition1}-\eqref{formulas:bimodule definition}.}\\
        (R2)\,\,\,\,&(ab)\, \eta(x)+(ab)\leftharpoonup \theta_{V}(x)=\theta(a)\,(b\leftharpoonup x)+\theta(a)\leftharpoonup(b\triangleright x);\\
        (R3)\,\,\,\,&(x\rightharpoonup a)\, \theta(b)+(x\triangleleft a)\rightharpoonup \theta(b)=\eta(x)\,(ab)+\theta_{V}(x)\rightharpoonup(ab);\\
        (R4)\,\,\,\,&(a\leftharpoonup x)\, \theta(b)+(a\triangleright x)\rightharpoonup \theta(b)=\theta(a)\,(x\rightharpoonup b)+\theta(a)\leftharpoonup(x\triangleleft b);\\
        (R5)\,\,\,\,&\eta(x)\,(y\rightharpoonup a)+\eta(x)\leftharpoonup(y\triangleleft a)+\theta_{V}(x)\rightharpoonup(y\rightharpoonup a)+f\big( \theta_{V}(x),y\triangleleft a\big)\\
        =&f(x,y)\,\theta(a)+(x\cdot_{V} y)\rightharpoonup\theta(a);\\
        (R6)\,\,\,\,&(x\cdot_{V} y)\triangleleft\theta(a)=\eta(x)\triangleright(y\triangleleft a)+\theta_{V}(x)\triangleleft(y\rightharpoonup a)+\theta_{V}(x)\cdot_{V}(y\triangleleft a);\\
        (R7)\,\,\,\,&(a\leftharpoonup x)\,\eta(y)+(a\leftharpoonup x)\leftharpoonup \theta_{V}(y)+(a\triangleright x)\rightharpoonup \eta(y)+f\big(a\triangleright x, \theta_{V}(y)\big)\\
        =&\theta(a)\,f(x,y)+ \theta(a)\leftharpoonup (x\cdot_{V} y) ;\\
        (R8)\,\,\,\,&(a\leftharpoonup x)\triangleright\theta_{V}(y)+(a\triangleright x)\triangleleft \eta(y)+(a\triangleright x)\cdot_{V} \theta_{V}(y)= \theta(a)\triangleright (x\cdot_{V} y) ;\\
        (R9)\,\,\,\,&(x\rightharpoonup a)\, \eta(y)+(x\rightharpoonup a)\leftharpoonup \theta_{V}(y)+(x\triangleleft a)\rightharpoonup \eta(y)+f\big(x\triangleleft a,\theta_{V}(y)\big)\\=
        &\eta(x)\,(a\leftharpoonup y)+\eta(x)\leftharpoonup(a\triangleright y)+ \theta_{V}(x)\rightharpoonup(a\leftharpoonup y)+f\big(\theta_{V}(x),a\triangleright y\big) ;\\
        (R10)\,\,&(x\rightharpoonup a)\triangleright \theta_{V}(y)+(x\triangleleft a)\triangleleft \eta(y)+(x\triangleleft a)\cdot_{V} \theta_{V}(y)\\=&\eta(x)\triangleright(a\triangleright y)
        +\theta_{V}(x)\triangleleft(a\leftharpoonup y)+ \theta_{V}(x)\cdot_{V}(a\triangleright y);\\
        (R11)\,\,&f(x,y)\,\eta(z)+f(x,y)\leftharpoonup\theta_{V}(z)+(x\cdot_{V} y)\rightharpoonup\eta(z)+f\big(x\cdot_{V} y,\theta_{V}(z)\big)\\=
        &\eta(x)\,f(y,z)+\eta(x)\leftharpoonup(y\cdot_{V} z)+\theta_{V}(x)\rightharpoonup f(y,z)+f\big(\theta_{V}(x),y\cdot_{V} z\big);\\
        (R12)\,\,&f(x,y)\triangleright\theta_{V}(z)+(x\cdot_{V} y)\triangleleft\eta(z)+(x\cdot_{V} y)\cdot_{V}\theta_{V}(z)\\=
        &\eta(x)\triangleright(y\cdot_{V} z)+\theta_{V}(x)\triangleleft f(y,z)+\theta_{V}(x)\cdot_{V} (y\cdot_{V} z);\\
        (R13)\,\,&P_{\alpha}(a)\,Q_{\beta}(x)+P_{\alpha}(a)\leftharpoonup P_{\beta,\,V}(x)=P_{\alpha\beta}\big(P_{\alpha}(a)\leftharpoonup x+a\,Q_{\beta}(x)\\
        &\,\,\,\,+a\leftharpoonup P_{\beta,\,V}(x)+\lambda\,a\leftharpoonup x\big)+Q_{\alpha\beta}\big(P_{\alpha}(a)\triangleright x+a\triangleright P_{\beta,\,V}(x)+\lambda\,a\triangleright x\big)\\
        (R14)\,\,&Q_{\alpha}(x)\,P_{\beta}(a)+P_{\alpha,\,V}(x)\rightharpoonup P_{\beta}(a)=P_{\alpha\beta}\big(Q_{\alpha}(x)\,a+P_{\alpha,\,V}(x)\rightharpoonup a\\
        &\,\,\,\,+x\rightharpoonup P_{\beta}(a)+\lambda\,x\rightharpoonup a\big)+Q_{\alpha\beta}\big(P_{\alpha,\,V}(x)\triangleleft a+x\triangleleft P_{\beta}(a)+\lambda\,x\triangleleft a\big)\\
        (R15)\,\,&Q_{\alpha}(x)\,Q_{\beta}(y)+Q_{\alpha}(x)\leftharpoonup P_{\beta,\,V}(y)+P_{\alpha,\,V}(x)\rightharpoonup Q_{\beta}(y)+f\big(P_{\alpha,\,V}(x),P_{\beta,\,V}(y)\big)\\
       = &P_{\alpha\beta}\Big(Q_{\alpha}(x)\leftharpoonup y+f\big(P_{\alpha,\,V}(x),y\big)+x\rightharpoonup Q_{\beta}(y)+f\big(x,P_{\beta,\,V}(y)\big)+\lambda\, f(x,y)\Big)\\
        &+Q_{\alpha\beta}\big(Q_{\alpha}(x)\triangleright y+P_{\alpha,\,V}(x)\cdot_{V} y+x\triangleleft Q_{\beta}(y)+x\cdot_{V} P_{\beta,\,V}(y)+\lambda\,x\cdot_{V} y\big);\\
        (R16)\,\,&Q_{\alpha}(x)\triangleright P_{\beta,\,V}(y)+P_{\alpha,\,V}(x)\triangleleft Q_{\beta}(y)+P_{\alpha,\,V}(x)\cdot_{V} P_{\beta,\,V}(y)\\=
        & P_{\alpha\beta,\,V}\big(Q_{\alpha}(x)\triangleright y+P_{\alpha,\,V}(x)\cdot_{V} y+x\triangleleft Q_{\beta}(y)+x\cdot_{V} P_{\beta,\,V}(y)+\lambda\,x\cdot_{V} y\big);\\
        (R17)\,\,&\theta\big(Q_{\alpha}(x)\big)+\eta\big(P_{\alpha,\,V}(x)\big)=P_{\alpha}\big(\eta(x)\big)+Q_{\alpha}\big(\theta_{V}(x)\big).
    \end{align*}
  \end{enumerate}
    \end{thm}

\begin{proof} The object $\rnv$ is a unified product if and only if $\Omega(R,V)$ is an extending datum such that the following conditions hold for all $a,b,c\in R$, $x,y,z\in V$, $\alpha,\beta\in S$:
        \begin{align}
    \big((a,x)\,\bar{\cdot}\, (b,y)\big)&\,\bar{\cdot}\, \bar{\theta}(c,z)=\bar{\theta}(a,x)\,\bar{\cdot}\, \big((b,y)\,\bar{\cdot}\, (c,z)\big);\label{zformula(not):unified product B}\\
    \bar{P}_{\alpha}(a,x) \,\bar{\cdot}\, \bar{P}_{\beta}(b,y)&=\bar{P}_{\alpha\beta}\big(\bar{P}_{\alpha}(a,x) \,\bar{\cdot}\, (b,y)+   (a,x)\,\bar{\cdot}\, \bar{P}_{\beta}(b,y)+\lambda\,(a,x) \,\bar{\cdot}\, (b,y)\big);\label{zformula(not):unified product RB}\\
    \bar{\theta}\big(\bar{P}_{\alpha}(a,x)\big)&=\bar{P}_{\alpha}\big(\bar{\theta}(a,x)\big).\label{zformula(not):unified product T}
    \end{align}

Since $\rnv$ is a direct sum of vector spaces $R$ and $V$, Eqs. ~\eqref{zformula(not):unified product B}-\eqref{zformula(not):unified product T} hold if and only if they hold for all generators of $\rnv$, i.e., for the set $\{(a,0)|a\in R\}\cup \{(0,x)|x\in V\}$. We have
\begin{align*}
    0=&\bar{P}_{\alpha}(a,0)\,\bar{\cdot}\, \bar{P}_{\beta}(0,x)-\bar{P}_{\alpha\beta}\big(\bar{P}_{\alpha}(a,0) \,\bar{\cdot}\, (0,x) +(a,0)\,\bar{\cdot}\, \bar{P}_{\beta}(0,x)+\lambda\,(a,0) \,\bar{\cdot}\, (0,x)\big)\\
 =&\big(P_{\alpha}(a),0\big) \,\bar{\cdot}\, \big(Q_{\beta}(x),P_{\beta,\,V}(x)\big)-\bar{P}_{\alpha\beta}\Big(\big(P_{\alpha}(a),0\big) \,\bar{\cdot}\, (0,x)+(a,0)\,\bar{\cdot}\, \big(Q_{\beta}(x),P_{\beta,\,V}(x)\big) \\
 &+\lambda\,(a\leftharpoonup x,a\triangleright x)\Big)\\
 =&\big(P_{\alpha}(a)\,Q_{\beta}(x)+P_{\alpha}(a)\leftharpoonup P_{\beta,\,V}(x),P_{\alpha}(a)\triangleright P_{\beta,\,V}(x)\big)-\bar{P}_{\alpha\beta}\big(P_{\alpha}(a)\leftharpoonup x+a\,Q_{\beta}(x)\\
        &+a\leftharpoonup P_{\beta,\,V}(x)+\lambda\,a\leftharpoonup x,P_{\alpha}(a)\triangleright x+a\triangleright P_{\beta,\,V}(x)+\lambda\,a\triangleright x \big)\\
 =&\big(P_{\alpha}(a)\,Q_{\beta}(x)+P_{\alpha}(a)\leftharpoonup P_{\beta,\,V}(x),P_{\alpha}(a)\triangleright P_{\beta,\,V}(x)\big)-\Big(P_{\alpha\beta}\big(P_{\alpha}(a)\leftharpoonup x+a\,Q_{\beta}(x)\\
        &+a\leftharpoonup P_{\beta,\,V}(x)+\lambda\,a\leftharpoonup x\big)+Q_{\alpha\beta}\big(P_{\alpha}(a)\triangleright x+a\triangleright P_{\beta,\,V}(x)+\lambda\,a\triangleright x\big),P_{\alpha\beta,\,V}\big(P_{\alpha}(a)\triangleright x\\
        &+a\triangleright P_{\beta,\,V}(x)+\lambda\,a\triangleright x\big)\Big).
\end{align*}
Hence,(R13) and Eq. ~\eqref{formulas:left module definition2} hold if and only if Eq. ~(\ref{zformula(not):unified product RB}) holds for the pair $(a,0)$, $(0,x)$ with $a\in R$, $x\in V$.

\begin{table}[htbp]
  \centering
\begin{tabular}{ccccc}
\toprule
\multicolumn{2}{c}{\phantom{12345678}condition 1}  & condition 2\\
\midrule
  \multirow{2}*{Eq. ~(\ref{zformula(not):unified product RB}) holds for the pair}&$(0,x)$, $(a,0)$
   &(R14) and Eq. ~\eqref{formulas:right module definition5} \\
      ~&$(0,x)$, $(0,y)$ &(R15) and (R16)\\

  \hline
 \multirow{7}*{Eq. ~(\ref{zformula(not):unified product B}) holds for the triple }&$(a,0)$, $(b,0)$, $(0,x)$
   &(R2) and Eq. ~\eqref{formulas:left module definition1}\\
      ~&$(0,x)$, $(a,0)$, $(b,0)$ &(R3) and Eq. ~\eqref{formulas:right module definition4}\\

      ~&$(a,0)$, $(0,x)$, $(b,0)$ &(R4) and Eq. ~\eqref{formulas:bimodule definition}\\

%
      ~&$(0,x)$, $(0,y)$, $(a,0)$ &(R5) and (R6)\\

      ~&$(a,0)$, $(0,x)$, $(0,y)$ &(R7) and (R8)\\

      ~&$(0,x)$, $(a,0)$, $(0,y)$ & (R9) and (R10)\\

      ~&$(0,x)$, $(0,y)$, $(0,z)$ & (R11) and (R12)\\

  \hline

 Eq. ~(\ref{zformula(not):unified product T}) holds for &$(0,x)$
   &(R17) and Eq. ~\eqref{formulas:left module definition3}\\

\bottomrule
\end{tabular}
\\[5pt]
  \caption{Equivalent conditions}\label{table:Equivalent conditions}
  \vspace{-1em}
\end{table}
Similarly, we have the specific equivalent conditions, refer to Table ~\ref{table:Equivalent conditions}. After tedious computation, we have proved the theorem.
 \end{proof}

According to Theorem ~\ref{thm:datum and unified product}, a Rota-Baxter family Hom extending structure of $\crr$ through $V$ will be viewed as an extending datum $\Omega(R,V)$ satisfying the conditions (R1)-(R17).

{\em Note:} In the rest of this paper, we will delete the trivial maps of $\Omega(R,V)$ for simplicity. For example, if the maps $\triangleleft$ and $\triangleright$ are both trivial, then the extending datum will be denoted by $\Omega(R,V)=(\rightharpoonup,\leftharpoonup,f,\cdot_{V},(Q_{\omega})_{\omega\in S},(P_{\omega,\,V})_{\omega\in S},\eta,\theta_{V})$ for simplicity. Now we will give some special cases of unified products as examples.

\begin{exam}\label{exam:direct sum and unified}
The extending datum $\Omega(R,V)=\big(\cdot_{V},(P_{\omega,\,V})_{\omega\in S},\theta_{V}\big)$ is a Rota-Baxter family Hom extending structure if and only if the 4-tuple $\big(V,\cdot_{V},(P_{\omega,\,V})_{\omega\in S},\theta_{V}\big)$ is a Rota-Baxter family Hom-associative algebra in Theorem ~\ref{thm:datum and unified product}. In this case, the corresponding unified product is a direct product of $\crr$ and $B$ as Rota-Baxter family Hom-associative algebras.
\end{exam}


\begin{exam}Let $\crr$ be a Rota-Baxter family Hom-associative algebra with Rota-Baxter family $(P_{\omega})_{\omega\in S}$ trivial, $V$ a vector space and $\Omega(R,V)$ an extending datum of $\crr$ through $V$.
\begin{enumerate}
\item Suppose that the semigroup $S=\{e\}$, then the triple $(R,\cdot,\theta)$ is a Hom-associative algebra and the Rota-Baxter family $(Q_{\omega})_{\omega\in S}$ is simply denoted by $Q$. Extending datum $\Omega(\crr,V)=(\triangleright,\triangleleft,Q,\theta_{V})$ is a Rota-Baxter family Hom extending structure if and only if the 4-tuple $\big(V,\theta_{V},\triangleright,\triangleleft\big)$ is a $(R,\cdot,\theta)$ bimodule, and for all $x,y\in V$, the following conditions hold:
\begin{align*}
Q(\theta_{V}(x))=\theta(Q(x)),\quad
Q(x)Q(y)=Q(Q(x)\triangleright y+x\triangleleft Q(y)).
 \end{align*}
i.e., $Q$ is an $O$ operator, more details refer to ~\cite[Definition~{\rm 2.3}]{CMM}. 
\item Suppose that $\theta=\mathrm{Id}$, then the pair $(R,\cdot)$ is an associative algebra. Extending datum $\Omega(\crr,V)=(\triangleright,\triangleleft,(Q_{\omega})_{\omega\in S},\theta_{V}=\mathrm{Id_{V}})$ is a Rota-Baxter family Hom extending structure  if and only if the triple $\big(V,\triangleright,\triangleleft\big)$ is $(R,\cdot)$ bimodule, and for all $x,y\in V$, the following condition holds:
\begin{align*}
Q_{\alpha}(x)Q_{\beta}(y)=Q_{\alpha\beta}(Q_{\alpha}(x)\triangleright y+x\triangleleft Q_{\beta}(y)),
 \end{align*}
i.e., the triple $(R,V,(Q_{\omega})_{\omega\in S})$ is a relative Rota-Baxter family algebra, more details refer to ~\cite[Definition~{\rm 2.2}]{Das}.
\end{enumerate}
\end{exam}


Let $\crr$ be a Rota-Baxter family Hom-associative algebra, $E$ a vector space containing $R$ as a subspace and $V$ a space complement of $R$ in $E$. Suppose that the Rota-Baxter family Hom extending structure $\Omega(R,V)\in \mathcal{O}(R,V)$ and $\rnv$ is the corresponding unified product. Similar to Definition~\ref{defn:equivalent cohomolous}, we define two relations $\equiv$ and $\approx$ between the unified product $\rnv$ and the extension $\ee\in \mathbf{Exts}(E,R)$ in the following Diagram ~(\ref{diagram:equivalent cohomologous ue}), and denoted by $\rnv\equiv \ee$ and $\rnv\approx \ee$ respectively.
\begin{align}\label{diagram:equivalent cohomologous ue}
\begin{split}
\xymatrix{
  0 \ar[r]^{} & R \ar[d]_{\mathrm {Id}} \ar[r]^{i^{\prime}\,\,\,\,\,} & \rnv \ar[d]_{\varphi} \ar[r]^{\,\,\,\,\,\pi^{\prime}} & V \ar[d]_{\mathrm {Id}} \ar[r]^{} & 0  \\
  0 \ar[r]^{} & R \ar[r]^{i\qquad\qquad} & \ee \ar[r]^{\qquad\qquad\pi} & V \ar[r]^{} & 0   }
\end{split}
\end{align}
where the maps $i,\pi$ are defined in Diagram ~(\ref{diagram:equivalent cohomologous ee}) and $i^{\prime},\pi^{\prime}$ are defined by $$i^{\prime}(a)=(a,0),\quad \pi^{\prime}(a,x)=x,\, a\in R,\, x\in V.$$

\begin{thm}\label{thm:extension to unified}
Let $\crr\subset \cee$ be an extension of Rota-Baxter family Hom-associative algebras. Then there exists a Rota-Baxter family Hom extending structure $\Omega(R,V)$ of $\crr$ through $V$, such that $\rnv\approx \cee$.
\end{thm}

\begin{proof}
 In the category $\mathbf{Vect}_{{\bf k}}$, there exists a linear retraction $\rho:\cee\rightarrow \crr$, i.e., a linear map $\rho$ satisfying $\rho\, i=\mathrm{Id}_{R}$. Then $V=\mathrm{ker}\,\rho$ is a space complement of $R$ in $E$.

First, for all $a\in R$, $x,y\in V$ and $\omega\in S$, define the extending datum $\Omega(R,V)$ as follows:
      \begin{align}\label{formulas:extension to unified}
      \begin{split}
        a\triangleright x&:=a\cdot_{E} x-\rho(a\cdot_{E} x);\\
        a\leftharpoonup x&:=\rho(a\cdot_{E} x);\\
        f(x,y)&:=\rho(x\cdot_{E} y);\\
        Q_{\omega}(x)&:=\rho\big(P_{\omega,\,E}(x)\big);\\
        \eta(x)&:=\rho\big(\theta_{E}(x)\big);
       \end{split}\qquad
      \begin{split}
        x\triangleleft a&:=x\cdot_{E} a-\rho(x\cdot_{E} a);\\
        x\rightharpoonup a&:=\rho(x\cdot_{E} a);\\
        x\cdot_{V} y&:=x\cdot_{E} y-\rho(x\cdot_{E} y);\\
        P_{\omega,\,V}(x)&:=P_{\omega,\,E}(x)-\rho\big(P_{\omega,\,E}(x)\big);\\
       \theta_{V}(x)&:=\theta_{E}(x)-\rho\big(\theta_{E}(x)\big).
       \end{split}
        \end{align}
Then we obtain Diagram ~\eqref{diagram:equivalent cohomologous ue}, where the map $\varphi$ is defined as follows:
\begin{align*}
\begin{split}
\varphi:\,\,\rnv&\rightarrow \big(E,\cdot_{E},(P_{\omega,\,E})_{\omega\in S},\theta_{E}\big)\\
(a,x)&\mapsto a+x,for\,\, all\,\, a \in R,x\in V.
\end{split}
\end{align*}
The linear map $\varphi$ is a linear isomorphism with the inverse map defined by $\varphi^{-1}(u)=(\rho(u),u-\rho(u))$, for all $u\in E$.

Second, we prove that the extending datum $\Omega(R,V)$ is a Rota-Baxter family Hom extending structure. By Theorem ~\ref{thm:datum and unified product}, we need to prove that $\rnv$ is a Rota-Baxter family Hom-associative algebra.
For all $a,b\in R$, $x,y \in V$, $\omega\in S$, we have
\begin{align}
\varphi\big((a,x)\,\bar{\cdot}\,(b,y)\big)
&=\varphi\big(ab+a\leftharpoonup y+x\rightharpoonup b+f(x,y),a\triangleright y+x\triangleleft b+x\cdot_{V} y\big)\quad\text{(by Eq.~(\ref{formulas:unified product}))}\nonumber\\
&=ab+a\leftharpoonup y+x\rightharpoonup b+f(x,y)+a\triangleright y+x\triangleleft b+x\cdot_{V} y\nonumber\\
&=ab+(a\leftharpoonup y+a\triangleright y)+(x\rightharpoonup b+x\triangleleft b)+\big(f(x,y)+x\cdot_{V} y\big)\nonumber\\
&=a\cdot_{E} b+a\cdot_{E} y+x\cdot_{E} b+x\cdot_{E} y\quad\text{(by Eq.~(\ref{formulas:extension to unified}))}\nonumber\\
&=(a+x)\cdot_{E}(b+y)\nonumber\\
&=\varphi(a,x)\cdot_{E}\varphi(b,y).\label{formulas:preserve multiplication1}\\
\varphi(\bar{P}_{\omega}(a,x))
&=\varphi\big(P_{\omega}(a)+Q_{\omega}(x),P_{\omega,\,V}(x)\big)\quad\text{(by Eq.~(\ref{formulas:unified product}))}\nonumber\\
&=P_{\omega}(a)+Q_{\omega}(x)+P_{\omega,\,V}(x)\nonumber\\
&=P_{\omega,\,E}(a)+P_{\omega,\,E}(x)\quad\text{(by Eq.~(\ref{formulas:extension to unified}))}\nonumber\\
&=P_{\omega,\,E}(a+x)\nonumber\\
&=P_{\omega,\,E}\big(\varphi(a,x)\big);\label{formulas:preserve multiplication2}\\
\varphi\big(\bar{\theta}(a,x)\big)
&=\varphi\big(\theta(a)+\eta(x),\theta_{V}(x)\big)\quad\text{(by Eq.~(\ref{formulas:unified product}))}\nonumber\\
&=\theta(a)+\eta(x)+\theta_{V}(x)\nonumber\\
&=\theta_{E}(a)+\theta_{E}(x)\quad\text{(by Eq.~(\ref{formulas:extension to unified}))}\nonumber\\
&=\theta_{E}(a+x)\nonumber\\
&=\theta_{E}\big(\varphi(a,x)\big).\label{formulas:preserve multiplication3}
\end{align}
Then we have
\begin{align*}
    &\big((a,x)\,\bar{\cdot}\, (b,y)\big)\,\bar{\cdot}\, \bar{\theta}(c,z)\\
    =&\varphi^{-1}\varphi\Big(\big((a,x)\,\bar{\cdot}\, (b,y)\big)\,\bar{\cdot}\, \bar{\theta}(c,z)\Big)\\
    =&\varphi^{-1}\Big(\big((a+x)\cdot_{E} (b+y)\big)\cdot_{E}\theta_{E}(c+z)\Big)\quad\text{(by Eq.~(\ref{formulas:preserve multiplication1}))}\\
    =&\varphi^{-1}\Big(\theta_{E}(a+x)\cdot_{E} \big((b+y)\cdot_{E} (c+z)\big)\Big)\\
    \quad&\text{(by $E$ being a Rota-Baxter family Hom-associative algebra})\\
    =&\varphi^{-1}\varphi\Big((a,x)\,\bar{\cdot}\, \big((b,y)\,\bar{\cdot}\, \bar{\theta}(c,z)\big)\Big)
       \quad\text{(by Eq.~(\ref{formulas:preserve multiplication1}))}\\
    =&\bar{\theta}(a,x)\,\bar{\cdot}\, \big((b,y)\,\bar{\cdot}\, (c,z)\big);\\
    &\bar{P}_{\alpha}(a,x) \,\bar{\cdot}\, \bar{P}_{\beta}(b,y)\\
    =&\varphi^{-1}\varphi\Big(\bar{P}_{\alpha}(a,x) \,\bar{\cdot}\, \bar{P}_{\beta}(b,y)\Big)\\
    =&\varphi^{-1}\Big(\varphi(\bar{P}_{\alpha}(a,x)) \,\cdot_{E}\, \varphi(\bar{P}_{\beta}(b,y))\Big)\quad\text{(by Eq.~(\ref{formulas:preserve multiplication1}))}\\
    =&\varphi^{-1}\Big(\bar{P}_{\alpha,\,E}(\varphi(a,x)) \,\cdot_{E}\, \bar{P}_{\beta,\,E}(\varphi(b,y))\Big)\quad\text{(by Eq.~(\ref{formulas:preserve multiplication2}))}\\
    =&\varphi^{-1}\Big(\bar{P}_{\alpha\beta,\,E}\Big(\bar{P}_{\alpha,\,E}(\varphi(a,x)) \,\cdot_{E}\, \varphi(b,y)+   \varphi(a,x)\,\cdot_{E}\, \bar{P}_{\beta,\,E}(\varphi(b,y))+\lambda\,\varphi(a,x) \,\cdot_{E}\, \varphi(b,y)\Big)\Big)\Big);\\
    \quad&\text{(by $E$ being a Rota-Baxter family Hom-associative algebra})\\
    =&\varphi^{-1}\Big(\bar{P}_{\alpha\beta,\,E}\Big(\varphi(\bar{P}_{\alpha}(a,x)) \,\cdot_{E}\, \varphi(b,y)+   \varphi(a,x)\,\cdot_{E}\, \varphi(\bar{P}_{\beta}(b,y))+\lambda\,\varphi(a,x) \,\cdot_{E}\, \varphi(b,y)\Big)\Big)\Big)\\
&\hspace{3cm}\text{(by Eq.~(\ref{formulas:preserve multiplication2}))}\\
    =&\varphi^{-1}\Big(\bar{P}_{\alpha\beta,\,E}\Big(\varphi\Big(\bar{P}_{\alpha}(a,x) \,\bar{\cdot}\, (b,y)+   (a,x)\,\bar{\cdot}\, \bar{P}_{\beta}(b,y)+\lambda\,(a,x) \,\bar{\cdot}\, (b,y)\Big)\Big)\Big)\quad\text{(by Eq.~(\ref{formulas:preserve multiplication1}))}\\
    =&\varphi^{-1}\varphi\Big(\bar{P}_{\alpha\beta}\Big(\bar{P}_{\alpha}(a,x) \,\bar{\cdot}\, (b,y) +   (a,x)\,\bar{\cdot}\, \bar{P}_{\beta}(b,y)+\lambda\,(a,x) \,\bar{\cdot}\, (b,y)\Big)\Big)\quad\text{(by Eq.~(\ref{formulas:preserve multiplication2}))}\\
    =&\bar{P}_{\alpha\beta}\Big(\bar{P}_{\alpha}(a,x) \,\bar{\cdot}\, (b,y) +   (a,x)\,\bar{\cdot}\, \bar{P}_{\beta}(b,y)+\lambda\,(a,x) \,\bar{\cdot}\, (b,y)\Big);\\
&\bar{\theta}\big(\bar{P}_{\alpha}(a,x)\big)\\
=&\varphi^{-1}\varphi\Big(\bar{\theta}\big(\bar{P}_{\alpha}(a,x)\big)\Big)\\
=&\varphi^{-1}\Big(\theta_{E}(P_{\alpha,\,E}(\varphi(a,x)))\Big)\quad\text{(by Eqs.~\eqref{formulas:preserve multiplication2}-\eqref{formulas:preserve multiplication3})}\\
=&\varphi^{-1}\Big(P_{\alpha,\,E}(\theta_{E}(\varphi(a,x)))\Big)\quad\text{(by $E$ being a Rota-Baxter family Hom-associative algebra})\\
=&\varphi^{-1}\varphi\Big(\bar{P}_{\alpha}(\bar{\theta}(a,x)\big)\Big)\quad\text{(by Eqs.~\eqref{formulas:preserve multiplication2}-\eqref{formulas:preserve multiplication3})}\\
&=\bar{P}_{\alpha}\big(\bar{\theta}(a,x)\big).
    \end{align*}
%


Finally, similar to ~\cite[Theorem 3.11]{ZW24}, we prove that the map $\varphi$ stabilizes $R$ and co-stabilizes $V$. Moreover, $\varphi$ is an isomorphism of Rota-Baxter family Hom-associative algebras by Eqs.~(\ref{formulas:preserve multiplication1})-(\ref{formulas:preserve multiplication3}).
\end{proof}

By the proof of Theorem ~\ref{thm:extension to unified}, we give the following remark.

\begin{rem}\label{rem:map unified to extension}
Let $\crr\subset \cee$ be an extension of Rota-Baxter family Hom-associative algebras.
\begin{enumerate}
\item\label{rem:map unified to extension1} Suppose that $\rho:E\rightarrow R$ is a linear retraction, then $V=\mathrm{ker}\,\rho$ is a space complement of $R$ in $E$. Conversely, suppose that $V$ is a space complement of $R$ in $E$, then we obtain a linear retraction $\rho:E\rightarrow R$ by defining $\rho(a+x)=a$, $a\in R,x\in V$, and it is called {\bf the retraction associated to $V$}~\cite{ZW24}.
\item\label{rem:map unified to extension2} Suppose that $\rho:E\rightarrow R$ is a linear retraction and $V=\mathrm{ker}\,\rho$, then we establishes a map as follows:
\begin{align*}
\Upsilon_{1}:\,\,\mathbf{Exts}(E,R)&\rightarrow \mathcal{O}(R,V)\\
\ee&\mapsto \Omega(R,V).
\end{align*}
\end{enumerate}
\end{rem}

In the sequel, we denote it by $\Omega^{\prime}(R,V)$ for convenience, the Rota-Baxter family Hom extending structure $\big(\triangleright^{\prime},\triangleleft^{\prime},\rightharpoonup^{\prime},
\leftharpoonup^{\prime},f^{\prime},\cdot_{V}^{\prime},(Q^{\prime}_{\omega})_{\omega\in S},(P^{\prime}_{\omega,\,V})_{\omega\in S},\eta^{\prime},\theta^{\prime}_{V}\big)$ of $\crr$ through $V$.

\begin{prop}\label{prop:bijection of extension and unified product}
The map $\Upsilon_{1}:\mathbf{Exts}(E,R)\rightarrow \mathcal{O}(R,V)$ defined in {\rm Remark ~\ref{rem:map unified to extension}~\ref{rem:map unified to extension2}} is bijective.
\end{prop}
\begin{proof}
Define a map
\begin{align*}
\Upsilon_{2}:\,\,\mathcal{O}(R,V)&\rightarrow \mathbf{Exts}(E,R)\\
\Omega(R,V)&\mapsto \ee,
\end{align*}
where the maps $\cdot_{E},P_{\omega,\,E}$ and $\theta_{E}$ are defined as follows with $a,b\in R$, $x,y\in V$ and $\omega\in S$:
\begin{align}\label{formulas:unified to extension}
\begin{split}
(a+x)\cdot_{E}(b+y)&:=ab+a\leftharpoonup y+x\rightharpoonup b+f(x,y)+a\triangleright y+x\triangleleft b+x\cdot_{V} y;\\
P_{\omega,\,E}(a,x)&:=P_{\omega}(a)+Q_{\omega}(x)+P_{\omega,\,V}(x);\\
\theta_{E}(a,x)&:=\theta(a)+\eta(x)+\theta_{V}(x).
\end{split}
\end{align}
Similar to ~\cite[Proposition 3.11]{ZW24}, we can prove that the map $\Upsilon_{2}$ is well defined, $\Upsilon_{1}\circ\Upsilon_{2}=\mathrm{id}_{\mathcal{O}(R,V)}$ and $\Upsilon_{2}\circ\Upsilon_{1}=\mathrm{id}_{\mathbf{Exts}(E,R)}$.
%
%
%
%
%
\end{proof}


\begin{defn}\label{defn:datum equivalent}
Let $\crr$ be a Rota-Baxter family Hom-associative algebra and $V$ a vector space. Two Rota-Baxter family Hom extending structures $\Omega(R,V)$ and $\Omega^{\prime}(R,V)$ are called {\bf equivalent}, and we denote it by $\Omega(R,V)\equiv\Omega^{\prime}(R,V)$, if there exists a pair of linear maps $(g,h)$, where $g:V\rightarrow R$ and $h\in \mathrm{Aut}_{{\bf k}}(V)$, satisfying the following conditions with $a\in R$, $x,y\in V$ and $\omega\in S$:
\begin{align*}
        (E1)\,\,&\text{The map $h$ is a $\crr$ bimodule morphism;}\\
        (E2)\,\,&a\leftharpoonup x+g(a\triangleright x)=a\,g(x)+a\leftharpoonup^{\prime}h(x);\\
        (E3)\,\,&x\rightharpoonup a+g(x\triangleleft a)=g(x)\,a+h(x)\rightharpoonup^{\prime} a ;\\
        (E4)\,\,&f(x,y)+g(x\cdot_{V} y)=g(x)\,g(y)+g(x)\leftharpoonup^{\prime} h(y)+h(x)\rightharpoonup^{\prime} g(y)+f^{\prime}\big(h(x),h(y)\big);\\
        (E5)\,\,&h(x\cdot_{V} y)=g(x)\triangleright^{\prime} h(y)+h(x)\triangleleft^{\prime} g(y)+h(x)\cdot_{V}^{\prime} h(y);\\
        (E6)\,\,&Q_{\omega}(x)+g\big(P_{\omega,\,V}(x)\big)=P_{\omega}\big(g(x)\big)+Q^{\prime}_{\omega}\big(h(x)\big);\\
        (E7)\,\,&\eta(x)+g\big(\theta_{V}(x)\big)=\theta\big(g(x)\big)+\eta^{\prime}\big(h(x)\big).
\end{align*}

Moreover, two equivalent Rota-Baxter family Hom extending structures are called cohomologous if $h=\mathrm{Id}_{V}$, and we denote it by $\Omega(R,V)\approx\Omega^{\prime}(R,V)$.
\end{defn}

Let $R$ be a Rota-Baxter family Hom-associative algebra and $V$ a vector space. Suppose that two Rota-Baxter family Hom extending structures $\Omega(R,V),\Omega^{\prime}(R,V)\in \mathcal{O}(R,V)$ and $\rnv,\rnvp$ are corresponding unified products. Similar to Definition~\ref{defn:equivalent cohomolous}, we define two relations $\equiv$ and $\approx$ between these two unified products $\rnv,\rnvp$ in the following Diagram ~(\ref{diagram:equivalent cohomologous uu}) and denoted by $\rnv\equiv \rnvp,\rnv\approx \rnvp$ respectively.
\begin{align}\label{diagram:equivalent cohomologous uu}
\begin{split}
\xymatrix{
  0 \ar[r]^{} & R \ar[d]_{\mathrm {Id}} \ar[r]^{i^{\prime}\,\,\,\,\,} & \rnv \ar[d]_{\phi} \ar[r]^{\,\,\,\,\,\pi^{\prime}} & V \ar[d]_{\mathrm {Id}} \ar[r]^{} & 0  \\
  0 \ar[r]^{} & R \ar[r]^{i^{\prime}\,\,\,\,\,} & \rnvp \ar[r]^{\,\,\,\,\,\pi^{\prime}} & V \ar[r]^{} & 0   }
\end{split}
\end{align}
where the maps $i^{\prime}$ and $\pi^{\prime}$ are defined in Diagram ~(\ref{diagram:equivalent cohomologous ue}).
\begin{lem}\label{lem:datum equivalent cohomologous}
Let $\crr$ be a Rota-Baxter family Hom-associative algebra and $V$ a vector space. Suppose that $\Omega(R,V)$ and $\Omega^{\prime}(R,V)$ are two Rota-Baxter family Hom extending structures of $\crr$ through $V$ and $\rnv$, $\rnvp$ are the corresponding unified products. We denote by $\mathcal{M}$ the set of all morphisms of Rota-Baxter family Hom-associative algebras $\phi:\rnv\rightarrow \rnvp$ which stabilizes $R$ and denote by $\mathcal{N}$ the set of $(g,h)$, where the maps $g:V\rightarrow R$ and $h:V\rightarrow V$ are both linear maps satisfying {\rm Eqs. (E1)-(E7)} in {\rm Definition ~\ref{defn:datum equivalent}}. Then
\begin{enumerate}
\item there exists a bijection
\begin{align*}
\Upsilon_{4}:\mathcal{N}&\rightarrow \mathcal{M}\\
(g,h)&\mapsto \phi
\end{align*}
where $\phi(a,x)=\big(a+g(x),h(x)\big)$, for all $a \in R,x \in V$.
\item under the bijection, the map $\phi$ is bijective if and only if the map $h$ is bijective and the map $\phi$ co-stabilizes $V$ if and only if $h=\mathrm{Id}_{V}$.
\end{enumerate}
\end{lem}


\begin{proof}
Analogous to ~\cite[Lemma 3.13]{ZW24}, here we only need to prove that Eqs. (E1)-(E7) hold if and only if
the map $\phi$ is a morphism of Rota-Baxter family Hom-associative algebras,
i.e., for all $a,b\in R$, $x,y\in V$ and $\omega\in S$, the following identities hold.
\begin{align}
   \phi\big((a,x)\,\bar{\cdot}\, (b,y)\big)&=\phi(a,x)\,\bar{\cdot}\,^{\prime} \phi(b,y);\label{zformula(not):algebra morphism M}\\
    \phi\big(\bar{P}_{\omega}(a,x)\big)&=\bar{P}^{\prime}_{\omega}\big(\phi(a,x)\big);\label{zformula(not):algebra morphism RB}\\
    \phi\big(\bar{\theta}(a,x)\big)&=\bar{\theta}^{\prime}\big(\phi(a,x)\big).\label{zformula(not):algebra morphism T}
\end{align}
By directly computing, Eqs. ~\eqref{zformula(not):algebra morphism M}-\eqref{zformula(not):algebra morphism T} hold if and only if Eqs. (E1)-(E7) hold.
\end{proof}


By Lemma ~\ref{lem:datum equivalent cohomologous}, we know that $\Omega(R,V)\equiv \Omega^{\prime}(R,V)$ (resp. $\Omega(R,V)\approx \Omega^{\prime}(R,V)$) if and only if $\rnv\equiv \rnvp$ (resp. $\rnv\approx \rnvp$) in {\rm Diagram ~(\ref{diagram:equivalent cohomologous uu})}. It's obvious that the relation $\equiv$ and $\approx$ are both equivalence relations. Suppose that
$$H^{2}(V,R):=\mathcal{O}(R,V)/\equiv,\quad H_{V}^{2}(V,R):=\mathcal{O}(R,V)/\approx.$$
Now we arrive at our main result of classifying the Rota-Baxter family Hom extending structures.

\begin{thm}\label{thm:classification of extension}
Let $\crr$ be a Rota-Baxter family Hom-associative algebra. Suppose that $E$ is a vector space containing $R$ as a subspace, $\rho:E\rightarrow R$ is a linear retraction and $V=\mathrm{ker}\,\rho$. Then
\begin{enumerate}
\item the map $\Upsilon_{1}$ defined in {\rm Remark ~\ref{rem:map unified to extension}~\ref{rem:map unified to extension2}} induces a bijection of equivalence classes via $\equiv$:
\begin{align*}
 \Upsilon:\mathbf{Extd}(E,R)&\rightarrow H^{2}(V,R);\\
 [\ee]&\mapsto [\Omega(R,V)].
 \end{align*}
 \item the map $\Upsilon_{1}$ induces a bijection of equivalence classes via $\approx$:
\begin{align*}
 \Upsilon^{\prime}:\mathbf{Extd^{\prime}}(E,R)&\rightarrow H_{V}^{2}(V,R);\\
 \overline{\ee}&\mapsto \overline{\Omega(R,V)}.
\end{align*}
\end{enumerate}
\end{thm}
\begin{proof}
It's a direct result of Proposition ~\ref{prop:induce map}.
\end{proof}

\subsection{Flag extending structures for Rota-Baxter family Hom-associative algebras}
In this subsection, we define flag datums as a special case of Rota-Baxter family Hom extending structures and give an example of the ES problem.
\begin{defn}
Let $\crr$ be a Rota-Baxter family Hom-associative algebra, $E$ a vector space containing $R$ as a subspace. A Rota-Baxter family Hom-associative algebra structure can be defined on $E$ is called a {\bf flag extending structure} of $\crr$ to $\cee$ if there exists a finite chain of subalgebras:
$$\crr=\cee_{0}\subset \cee_{1}\subset\cdots\subset \cee_{n}=\cee$$
such that $E_{i}$ has codimension 1 in $E_{i+1}$, for $0\leq i \leq n-1$.
\end{defn}

Suppose that $R$ has finite codimension in $E$ and $V$ is a space complement of $\crr$ in $E$ with basis $\{x_{1},x_{2},\cdots,x_{n}\}$, then the flag extending structure of Rota-Baxter family Hom-associative algebra $R$ to $E$ can be defined recursively. In fact, we can first define all extensions on $E_{1}:=R+{\bf k}\{x_{1}\}$ through Rota-Baxter family Hom extending structures $\Omega(R,{\bf k}\{x_{1}\})$. Second, we can define all extensions on $E_{2}:=E_{1}+{\bf k}\{x_{2}\}$ through Rota-Baxter family Hom extending structures $\Omega(E_{1},{\bf k}\{x_{2}\}),\cdots$, by finite steps, we can define all extensions on $E_{n}:=E_{n-1}+{\bf k}\{x_{n}\}$. Since each step is similar to the first one, we mainly consider the Rota-Baxter family Hom extending structure of $R$ through a 1-dimensional vector space $V$.
\begin{defn}
Let $\crr$ be a Rota-Baxter family Hom-associative algebra. Suppose that $l,r:R\rightarrow {\bf k}$ and $t_{r},t_{l}:R\rightarrow R$ are linear maps, $a_{1},(b_{\omega})_{\omega\in S},a_{2}\in R$ and $\bar{k}_{1},(\bar{k}_{\omega})_{\omega\in S},\bar{k}_{2}\in {\bf k}$. The 10-tuple $\Omega(\crr)=\big(l,r,t_{r},t_{l},a_{1},\bar{k}_{1},(b_{\omega})_{\omega\in S},(\bar{k}_{\omega})_{\omega\in S},a_{2},\bar{k}_{2}\big)$ is called a {\bf flag datum} of $\crr$ if the following conditions hold for all $a,b\in R$:
\begin{align*}
        (F1)\,\,\,\,&l(ab)\,\bar{k}_{2}=l\big(\theta(a)\big)\,l(b);\quad
        (F2)\,\,\,\,r(a)\,r\big(\theta(b)\big)=r(ab)\,\bar{k}_{2};\quad
        (F3)\,\,\,\,l(a)\,r(\theta(b))=l\big(\theta(a)\big)\,r(b);\\
        (F4)\,\,\,\,&l\big(P_{\alpha}(a)\big)\,\bar{k}_{\beta}=\bar{k}_{\alpha\beta}\,\Big(l\big(P_{\alpha}(a)\big)+l(a)\,\bar{k}_{\beta}+\lambda\, l(a)\Big);\\
        (F5)\,\,\,\,&\bar{k}_{\alpha}\,r\big(P_{\beta}(a)\big)=\bar{k}_{\alpha\beta}\,\Big(\bar{k}_{\alpha}r(a)+r\big(P_{\beta}(a)\big)+\lambda\, r(a)\Big);\\
        (F6)\,\,\,\,&(ab)\,a_{2}+t_{l}(ab)\,\bar{k}_{2}=\theta(a)\,t_{l}(b)+t_{l}\big(\theta(a)\big)\,l(b);\\
        (F7)\,\,\,\,&t_{r}(a)\,\theta(b)+r(a)\,t_{r}\big(\theta(b)\big)=a_{2}\,(ab)+\bar{k}_{2}\,t_{r}(ab);\\
        (F8)\,\,\,\,&t_{l}(a)\,\theta(b)+l(a)\,t_{r}\big(\theta(b)\big)=\theta(a)\,t_{r}(b)+t_{l}\big(\theta(a)\big)\,r(b);\\
        (F9)\,\,\,\,&a_{1}\,\theta(a)+\bar{k}_{1}\,t_{r}\big(\theta(a)\big)=a_{2}\,t_{r}(a)+t_{l}(a_{2})\,r(a)+\bar{k}_{2}\,t_{r}\big(t_{r}(a)\big)+a_{1}\,\bar{k}_{2}\,r(a);\\
        (F10)\,\,&\bar{k}_{1}\,r\big(\theta(a)\big)=l(a_{2})\,r(a)+\bar{k}_{2}\,r\big(t_{r}(a)\big)+\bar{k}_{1}\,\bar{k}_{2}\,r(a);\\
        (F11)\,\,&t_{l}(a)\,a_{2}+t_{l}\big(t_{l}(a)\big)\,\bar{k}_{2}+l(a)\,t_{r}(a_{2})+l(a)\,\bar{k}_{2}\,a_{1}= \theta(a)\,a_{1}+ t_{l}\big(\theta(a)\big)\,\bar{k}_{1};\\
        (F12)\,\,&l\big(t_{l}(a)\big)\,\bar{k}_{2}+l(a)\,r(a_{2})+l(a)\,\bar{k}_{1}\,\bar{k}_{2}= l\big(\theta(a)\big)\,\bar{k}_{1} ;\\
        (F13)\,\,&t_{r}(a)\,a_{2}+t_{l}\big(t_{r}(a)\big)\,\bar{k}_{2}+r(a)\,t_{r}(a_{2})+r(a)\,\bar{k}_{2}\,a_{1}=a_{2}\,t_{l}(a)+t_{l}(a_{2})\,l(a)+\bar{k}_{2}\,t_{r}\big(t_{l}(a)\big)+\bar{k}_{2}\,l(a)\,a_{1};\\
        (F14)\,\,&l\big(t_{r}(a)\big)\,\bar{k}_{2}+r(a)\,r(a_{2})+r(a)\,\bar{k}_{1}\,\bar{k}_{2}=l(a_{2})\,l(a)+\bar{k}_{2}\,r\big(t_{l}(a)\big)+\bar{k}_{2}\,\bar{k}_{1}\,l(a);\\
        (F15)\,\,&a_{1}\,a_{2}+t_{l}(a_{1})\,\bar{k}_{2}+\bar{k}_{1}\,t_{r}(a_{2})=a_{2}\,a_{1}+\bar{k}_{1}\,t_{l}(a_{2})+\bar{k}_{2}\,t_{r}(a_{1});\\
        (F16)\,\,&l(a_{1})\,\bar{k}_{2}+\bar{k}_{1}\,r(a_{2})=l(a_{2})\,\bar{k}_{1}+\bar{k}_{2}\,r(a_{1});\\
        (F17)\,\,&P_{\alpha}(a)\,b_{\beta}+t_{l}\big(P_{\alpha}(a)\big)\,\bar{k}_{\beta}=P_{\alpha\beta}\Big(t_{l}\big(P_{\alpha}(a)\big)+a\,b_{\beta}+t_{l}(a)\,\bar{k}_{\beta}+\lambda\, t_{l}(a)\Big)+b_{\alpha\beta}\,\Big(l\big(P_{\alpha}(a)\big)\\
        &+l(a)\,\bar{k}_{\beta}+\lambda\,l(a)\Big);\\
        (F18)\,\,&b_{\alpha}\,P_{\beta}(a)+\bar{k}_{\alpha}\,t_{r}\big( P_{\beta}(a)\big)=P_{\alpha\beta}\Big(b_{\alpha}\,a+\bar{k}_{\alpha}\,t_{r}(a)+t_{r}\big(P_{\beta}(a)\big)+\lambda\, t_{r}(a)\Big)+b_{\alpha\beta}\,\Big(\bar{k}_{\alpha}\,r(a)\\
        &+r\big(P_{\beta}(a)\big)+\lambda\,r(a)\Big);\\
        (F19)\,\,&b_{\alpha}\,b_{\beta}+t_{l}(b_{\alpha})\,\bar{k}_{\beta}+\bar{k}_{\alpha}\,t_{r}(b_{\beta})+a_{1}\,\bar{k}_{\alpha}\,\bar{k}_{\beta}= P_{\alpha\beta}\big(t_{l}(b_{\alpha})+\bar{k}_{\alpha}\,a_{1}+t_{r}(b_{\beta})+\bar{k}_{\beta}\,a_{1}+\lambda\,a_{1}\big)\\ &+b_{\alpha\beta}\,\big(l(b_{\alpha})+\bar{k}_{\alpha}\,\bar{k}_{1}+r(b_{\beta})+\bar{k}_{1}\,\bar{k}_{\beta}+\lambda\,\bar{k}_{1}\big);\\
        (F20)\,\,&l(b_{\alpha})\,\bar{k}_{\beta}+\bar{k}_{\alpha}\,r(b_{\beta})+\bar{k}_{1}\,\bar{k}_{\alpha}\,\bar{k}_{\beta}= \bar{k}_{\alpha\beta}\,\big(l(b_{\alpha})+\bar{k}_{1}\,\bar{k}_{\alpha}+r(b_{\beta})+\bar{k}_{1}\,\bar{k}_{\beta}+\lambda\,\bar{k}_{1}\big);\\
        (F21)\,\,&\theta(b_{\alpha})+\bar{k}_{\alpha}\,a_{2}=P_{\alpha}(a_{2})+\bar{k}_{2}\,b_{\alpha}.
    \end{align*}
\end{defn}

Denote by $\mathcal{F}(\crr)$ the set of all flag datums of $\crr$, then we have the following theorem.

\begin{thm}\label{thm:flag to extending}
Let $\crr$ be a Rota-Baxter family Hom-associative algebra of codimension $1$ in $E$. Suppose that $V$ is a space complement of $R$ in $E$ with basis $\{x\}$. Then there exists a bijection:
\begin{align*}
 \Phi:\mathcal{F}(\crr)&\rightarrow \mathcal{O}(\crr,V)\\
 \Omega(R)&\mapsto \Omega(R,V),
 \end{align*}
where the Rota-Baxter family Hom extending structure $\Omega(R,V)$ is defined as follows with $a\in R$:
\begin{align}
\label{formulas:flag and extending structure}
\begin{aligned}
        a\triangleright x&=l(a)\,x;\\
        f(x,x)&=a_{1};\\
        \eta(x)&=a_{2};
\end{aligned}\quad
\begin{aligned}
        x\triangleleft a&=r(a)\,x;\\
        x\cdot_{V} x&=\bar{k}_{1}\,x;\\
        \theta_{V}(x)&=\bar{k}_{2}\,x.
\end{aligned}\quad
\begin{aligned}
        a\leftharpoonup x&=t_{l}(a);\\
        Q_{\omega}(x)&=b_{\omega};\\
        &
\end{aligned}\quad
\begin{aligned}
        x\rightharpoonup a&=t_{r}(a);\\
        P_{\omega,\,V}(x)&=\bar{k}_{\omega}\,x;\\
        &
\end{aligned}
\end{align}
\end{thm}

\begin{proof}
We can directly check that Eqs. (F1)-(F21) of the flag datum are equivalent to Eqs. (R1)-(R17) of the Rota-Baxter family Hom extending structure, hence $\Phi$ is a bijection.
\end{proof}

For better understanding, now we give an example to compute $H^{2}(V,\crr)$ and $H_{V}^{2}(V,\crr)$. Since the classification depends on the choice of field {\bf k}, we select {\bf k} the the complex number field in the following example.

\begin{exam}\label{exam:ES problem}
\begin{table}[htbp]
  \centering
  \renewcommand\arraystretch{1.1}
\begin{tabular}{ccccc}
\toprule
  \multirow{2}*{}&flag datums $(\bar{l},\bar{r},\bar{t}_{r},\bar{t}_{l},\bar{a}_{1},$
   & \multirow{2}*{$\bar{h}$} & \multirow{2}*{$\bar{g}$} & equivalent or \\
      ~&$\bar{k}_{1}, (\bar{b}_{e},\bar{b}_{\sigma}),(\bar{k}_{e},\bar{k}_{\sigma}),\bar{a}_{2},
\bar{k}_{2})$ & ~& ~ & cohomologous class\\

\midrule
  \multirow{2}*{1}&$\big(0,0,\bar{t}_{r},\bar{t}_{r},\bar{t}^{2}_{r},0,(-\bar{t}_{r}\bar{k}_{e},
  -\bar{t}_{r}\bar{k}_{\sigma}),$
   & \multirow{2}*{1} & \multirow{2}*{$-\bar{t}_{r}$} & \multirow{2}*{$\big(0_{6},(0_{2}),(\bar{k}_{e},\bar{k}_{\sigma}),0,
  \bar{k}_{2}\big)$} \\
      ~&$(\bar{k}_{e},\bar{k}_{\sigma}),
  (1-\bar{k}_{2})\bar{t}_{r},\bar{k}_{2}\big)$ & ~& ~ & ~\\

      \hline
   \multirow{2}*{2}& $\big(0,0,\bar{t}_{r},\bar{t}_{r},(\bar{t}_{r}-\bar{k}_{1})\bar{t}_{r},\bar{k}_{1},(0_{2})_{2},$ & $\frac{1}{\bar{k}_{1}}$ ~& $-\frac{1}{\bar{k}_{1}}\bar{t}_{r}$ & $\big(0_{5},1,(0_{2})_{2},0,\bar{k}_{2}\big)$  \\
      ~&$(1-\bar{k}_{2})\bar{t}_{r},\bar{k}_{2}\big)$, $\bar{k}_{1}\neq 0$ & $1$ & $-\bar{t}_{r}$ & $\big(0_{5},\bar{k}_{1},(0_{2})_{2},0,\bar{k}_{2}\big)$\\
    \hline

   \multirow{2}*{3}& $\big(0,0,\bar{t}_{r},\bar{t}_{r},(\bar{t}_{r}-\bar{k}_{1})\bar{t}_{r},\bar{k}_{1},(\lambda \bar{k}_{1},\pm \lambda \bar{k}_{1}),$ & $\frac{1}{\bar{k}_{1}}$ ~& $-\frac{1}{\bar{k}_{1}}\bar{t}_{r}$ & $\big(0_{5},1,(\lambda,\pm \lambda),(0_{2}),0,1\big)$  \\
      ~&$(0_{2}),0,1\big)$, $\bar{k}_{1}\neq 0$ & $1$ & $-\bar{t}_{r}$ & $\big(0_{5},\bar{k}_{1},(\lambda \bar{k}_{1},\pm \lambda \bar{k}_{1}),(0_{2}),0,1\big)$\\

  \hline
    \multirow{2}*{4}& $\big(0,0,\bar{t}_{r},\bar{t}_{r},(\bar{t}_{r}-\bar{k}_{1})\bar{t}_{r},\bar{k}_{1},(\lambda \bar{t}_{r},\lambda \bar{t}_{r}),$ & $\frac{1}{\bar{k}_{1}}$ ~& $-\frac{1}{\bar{k}_{1}}\bar{t}_{r}$ & $\big(0_{5},1,(0_{2}),(-\lambda,-\lambda),0,\bar{k}_{2}\big)$  \\
      ~&$(-\lambda,-\lambda),(1-\bar{k}_{2})\bar{t}_{r},\bar{k}_{2}\big)$, $\bar{k}_{1}\neq 0$ & $1$ & $-\bar{t}_{r}$ & $\big(0_{5},\bar{k}_{1},(0_{2}),(-\lambda,-\lambda),0,\bar{k}_{2}\big)$\\

\hline
    \multirow{2}*{5}&$\big(0,0,\bar{t}_{r},\bar{t}_{r},(\bar{t}_{r}-\bar{k}_{1})\bar{t}_{r},\bar{k}_{1},(\lambda \bar{t}_{r}-\bar{k}_{1}\lambda,$ & $\frac{1}{\bar{k}_{1}}$ ~& $-\frac{1}{\bar{k}_{1}}\bar{t}_{r}$ &$\big(0_{5},1,(-\lambda,\pm \lambda),(-\lambda,-\lambda),0,1\big)$ \\
      ~&$\lambda \bar{t}_{r}\pm \bar{k}_{1}\lambda),(-\lambda,-\lambda),0,1\big)$, $\bar{k}_{1}\neq 0$ & $1$ & $-\bar{t}_{r}$ & $\big(0_{5},\bar{k}_{1},(-\lambda \bar{k}_{1},\pm\lambda \bar{k}_{1}),(-\lambda,-\lambda),0,1\big)$\\

  \hline
  \multirow{2}*{6}&$\big(0,\bar{k}_{2},(1-\bar{k}_{2})\bar{t}_{l},\bar{t}_{l},(1-\bar{k}_{2})\bar{t}^{2}_{l},\bar{k}_{2}\bar{t}_{l},(0_{2})_{2},$
   & \multirow{2}*{1} & \multirow{2}*{$-\bar{t}_{l}$} & \multirow{2}*{$\big(0,\bar{k}_{2},0_{4},(0_{2})_{2},0,\bar{k}_{2}\big)$} \\
      ~&$(1-\bar{k}_{2})\bar{t}_{l},\bar{k}_{2}\big)$, $\bar{k}_{2}\neq 0$ & ~& ~ & ~\\

  \hline
  \multirow{2}*{7}&$\big(0,\bar{k}_{2},(1-\bar{k}_{2})\bar{t}_{l},\bar{t}_{l},(1-\bar{k}_{2})\bar{t}^{2}_{l},\bar{k}_{2}\bar{t}_{l},(\lambda \bar{t}_{l},$
   & \multirow{2}*{1} & \multirow{2}*{$-\bar{t}_{l}$} & \multirow{2}*{$\big(0,\bar{k}_{2},0_{4},(0_{2}),(-\lambda,-\lambda),0,\bar{k}_{2}\big)$} \\
      ~&$\lambda \bar{t}_{l}),(-\lambda,-\lambda),(1-\bar{k}_{2})\bar{t}_{l},\bar{k}_{2}\big)$, $\bar{k}_{2}\neq 0,1$ & ~& ~ & ~\\

\hline
    \multirow{2}*{8}&$\big(0,1,0,\bar{t}_{l},0,\bar{t}_{l},(\lambda \bar{t}_{l}+b_{0},\lambda \bar{t}_{l}+b_{0}),$ & $\frac{1}{b_{0}}$ ~& $-\frac{1}{b_{0}}\bar{t}_{l}$ & $\big(0,1,0_{4},(1,1),(-\lambda,-\lambda),0,1\big)$\\
      ~&$(-\lambda,-\lambda),0,1\big)$, $b_{0}\neq 0$ & $1$ & $-\bar{t}_{l}$ & $\big(0,1,0_{4},(b_{0},b_{0}),(-\lambda,-\lambda),0,1\big)$\\

\hline
    \multirow{2}*{9}&$\big(0,1,0,\bar{t}_{l},0,\bar{t}_{l},(\lambda \bar{t}_{l}+b_{0},\lambda \bar{t}_{l}),$ & $\frac{1}{b_{0}}$ ~& $-\frac{1}{b_{0}}\bar{t}_{l}$ & $\big(0,1,0_{4},(1,0),(-\lambda,-\lambda),0,1\big)$\\
      ~&$(-\lambda,-\lambda),0,1\big)$, $b_{0}\neq 0$ & $1$ & $-\bar{t}_{l}$ & $\big(0,1,0_{4},(b_{0},0),(-\lambda,-\lambda),0,1\big)$\\

  \hline
  \multirow{2}*{10}&$\big(\bar{k}_{2},0,\bar{t}_{r},(1-\bar{k}_{2})\bar{t}_{r},
  (1-\bar{k}_{2})\bar{t}^{2}_{r},\bar{k}_{2}\bar{t}_{r},$
   & \multirow{2}*{1} & \multirow{2}*{$-\bar{t}_{r}$} & \multirow{2}*{$\big(\bar{k}_{2},0_{5},(0_{2})_{2},0,\bar{k}_{2}\big)$} \\
      ~&$(0_{2})_{2},(1-\bar{k}_{2})\bar{t}_{r},\bar{k}_{2}\big)$, $\bar{k}_{2}\neq 0$ & ~& ~ & ~\\

  \hline
  \multirow{2}*{11}&$\big(\bar{k}_{2},0,\bar{t}_{r},(1-\bar{k}_{2})\bar{t}_{r},(1-\bar{k}_{2})\bar{t}^{2}_{r},\bar{k}_{2}\bar{t}_{r},(\lambda \bar{t}_{r},$
   & \multirow{2}*{1} & \multirow{2}*{$-\bar{t}_{r}$} & \multirow{2}*{$\big(\bar{k}_{2},0_{5},(0_{2}),(-\lambda,-\lambda),0,\bar{k}_{2}\big)$} \\
      ~&$\lambda \bar{t}_{r}),(-\lambda,-\lambda),(1-\bar{k}_{2})\bar{t}_{r},\bar{k}_{2}\big)$, $\bar{k}_{2}\neq 0,1$ & ~& ~ & ~\\

\hline
    \multirow{2}*{12}&$\big(1,0,\bar{t}_{r},0,0,\bar{t}_{r},(\lambda \bar{t}_{r}+b_{0},\lambda \bar{t}_{r}),$ & $\frac{1}{b_{0}}$ ~& $-\frac{1}{b_{0}}\bar{t}_{r}$ & $\big(1,0_{5},(1,0),(-\lambda,-\lambda),0,1\big)$\\
      ~&$(-\lambda,-\lambda),0,1\big)$, $b_{0}\neq 0$ & $1$ & $-\bar{t}_{r}$ & $\big(1,0_{5},(b_{0},0),(-\lambda,-\lambda),0,1\big)$\\

\hline
    \multirow{2}*{13}&$\big(1,0,\bar{t}_{r},0,0,\bar{t}_{r},(\lambda \bar{t}_{r}+b_{0},\lambda \bar{t}_{r}+b_{0}),$ & $\frac{1}{b_{0}}$ ~& $-\frac{1}{b_{0}}\bar{t}_{r}$ & $\big(1,0_{5},(1,1),(-\lambda,-\lambda),0,1\big)$\\
      ~&$(-\lambda,-\lambda),0,1\big)$, $b_{0}\neq 0$ & $1$ & $-\bar{t}_{r}$ & $\big(1,0_{5},(b_{0},b_{0}),(-\lambda,-\lambda),0,1\big)$\\

  \hline
  \multirow{2}*{14}&\multirow{2}*{$\big(1,1,0,0,\frac{-\bar{k}_{1}^{2}}{4},\bar{k}_{1},(0_{2})_{2},0,1\big)$}
   & \multirow{2}*{1} & \multirow{2}*{$-\frac{\bar{k}_{1}}{2}$} & \multirow{2}*{$\big(1,1,0_{4},(0_{2})_{2},0,1\big)$} \\
      ~&~ & ~& ~ & ~\\


\hline
    \multirow{2}*{15}&\multirow{2}*{$\big(1,1,0,0,\bar{a}_{1}-\frac{\bar{k}_{1}^{2}}{4},\bar{k}_{1},(0_{2})_{2},0,1\big)$, $\bar{a}_{1}\neq 0$} & $\frac{1}{\sqrt{\bar{a}_{1}}}$& $-\frac{\bar{k}_{1}}{2\sqrt{\bar{a}_{1}}}$ & $\big(1,1,0,0,1,0,(0_{2})_{2},0,1\big)$\\
      ~&~ & $1$ & $-\frac{\bar{k}_{1}}{2}$ & $\big(1,1,0,0,\bar{a}_{1},0,(0_{2})_{2},0,1\big)$\\

  \hline
  \multirow{2}*{16}&$\big(1,1,0,0,-b_{0}^{2},2b_{0},(\lambda b_{0},\lambda b_{0}),$
   & \multirow{2}*{1} & \multirow{2}*{$-b_{0}$} & \multirow{2}*{$\big(1,1,0_{4},(0_{2}),(-\lambda,-\lambda),0,1\big)$} \\
      ~&$(-\lambda,-\lambda),0,1\big)$ & ~& ~ & ~\\

\hline
    \multirow{2}*{17}&$\big(1,1,0,0,-b_{0}(b_{0}+\bar{k}_{1}),\bar{k}_{1}+2b_{0},$ & $\frac{1}{\bar{k}_{1}}$ ~& $-\frac{1}{\bar{k}_{1}}b_{0}$ & $\big(1,1,0_{3},1,(0_{2}),(-\lambda,-\lambda),0,1\big)$\\
      ~&$(\lambda b_{0},\lambda b_{0}),(-\lambda,-\lambda),0,1\big)$, $\bar{k}_{1}\neq 0$ & $1$ & $-b_{0}$ & $\big(1,1,0_{3},\bar{k}_{1},(0_{2}),(-\lambda,-\lambda),0,1\big)$\\

\hline
    \multirow{2}*{18}&$\big(-1,-1,\bar{a}_{2},\bar{a}_{2},\frac{3\bar{a}^{2}_{2}}{4}+a^{2}_{0},-\bar{a}_{2},(0_{2})_{2},$ & $\frac{1}{\bar{a}_{0}}$ ~& $-\frac{\bar{a}_{2}}{2\bar{a}_{0}}$ & $\big(-1,-1,0,0,1,0,(0_{2})_{2},0,-1\big)$\\
      ~&$\bar{a}_{2},-1\big)$, $\bar{a}_{0}\neq 0$ & $1$ & $-\frac{\bar{a}_{2}}{2}$ & $\big(-1,-1,0,0,a^{2}_{0},0,(0_{2})_{2},0,-1\big)$\\

  \hline
  \multirow{2}*{19}&$\big(\bar{k}_{2},\bar{k}_{2},\bar{a}_{2},\bar{a}_{2},\frac{(1-2 \bar{k}_{2})\bar{a}^{2}_{2}}{(1-\bar{k}_{2})^{2}},\frac{2 \bar{k}_{2}\bar{a}_{2}}{1-\bar{k}_{2}},(0_{2})_{2},$
   & \multirow{2}*{1} & \multirow{2}*{$-\frac{\bar{a}_{2}}{1-\bar{k}_{2}}$} & \multirow{2}*{$\big(\bar{k}_{2},\bar{k}_{2},0_{4},(0_{2})_{2},0,\bar{k}_{2}\big)$} \\
      ~&$\bar{a}_{2},\bar{k}_{2}\big)$, $\bar{k}_{2}\neq 0,\pm 1$ & ~& ~ & ~\\

  \hline
  \multirow{2}*{20}&$\big(\bar{k}_{2},\bar{k}_{2},\bar{a}_{2},\bar{a}_{2},\frac{(1-2 \bar{k}_{2})\bar{a}^{2}_{2}}{(1-\bar{k}_{2})^{2}},\frac{2 \bar{k}_{2}\bar{a}_{2}}{1-\bar{k}_{2}},(\frac{\bar{a}_{2}\lambda}{1-\bar{k}_{2}},$
   & \multirow{2}*{1} & \multirow{2}*{$-\frac{\bar{a}_{2}}{1-\bar{k}_{2}}$} & \multirow{2}*{$\big(\bar{k}_{2},\bar{k}_{2},0_{4},(0_{2}),(-\lambda,-\lambda),0,\bar{k}_{2}\big)$} \\
      ~&$\frac{\bar{a}_{2}\lambda}{1-\bar{k}_{2}}),(-\lambda,-\lambda),\bar{a}_{2},\bar{k}_{2}\big)$, $\bar{k}_{2}\neq 0,1$ & ~& ~ & ~\\
\bottomrule
\end{tabular}
\\[5pt]
  \caption{Classifying flag datums}\label{table:classifying flag datum}
  \vspace{-1em}
\end{table}

To continue Example~\ref{exam:dendriform algebra}~\ref{exam:dendriform algebra1} and suppose $V={\bf k}\{e_{2}\}$.
Let  $\big(l,r,t_{r},t_{l},a_{1},\bar{k}_{1},(b_{\omega})_{\omega\in S},$ $(\bar{k}_{\omega})_{\omega\in S},a_{2},\bar{k}_{2}\big)$ be a flag datum of $\crr$, and $(g,h)$ the pair of linear maps defined in Definition ~\ref{defn:datum equivalent}. Suppose that
\begin{align}\label{formulas:flag and D}
\begin{aligned}
\begin{split}
        l(e_{1})&=\bar{l};\\
        a_{1}&=\bar{a}_{1}e_{1};\\
        g(e_{2})&=\bar{g}e_{1};
\end{split}
\end{aligned}
\begin{aligned}
\begin{split}
        r(e_{1})&=\bar{r};\\
        b_{e}&=\bar{b}_{e}e_{1};\\
        h(e_{2})&=\bar{h}e_{2};
\end{split}
\end{aligned}
\begin{aligned}
\begin{split}
        t_{l}(e_{1})&=\bar{t}_{l}e_{1};\\
        b_{\sigma}&=\bar{b}_{\sigma}e_{1};\\
        a&=b=e_{1};
\end{split}
\end{aligned}
\begin{aligned}
\begin{split}
        t_{r}(e_{1})&=\bar{t}_{r}e_{1};\\
        a_{2}&=\bar{a}_{2}e_{1};\\
        &
\end{split}
\end{aligned}
\end{align}
where elements $\bar{l},\bar{r},\bar{t}_{r},\bar{t}_{l},
\bar{a}_{1},\bar{k}_{1},(\bar{b}_{e},\bar{b}_{\sigma}),(\bar{k}_{e},\bar{k}_{\sigma}),
\bar{a}_{2},\bar{k}_{2},\bar{g},\bar{h}\in {\bf k}$, and the 10-tuple $(\bar{l},\bar{r},\bar{t}_{r},\bar{t}_{l},\bar{a}_{1},\bar{k}_{1},$ $
(\bar{b}_{e},\bar{b}_{\sigma}),(\bar{k}_{e},\bar{k}_{\sigma}),\bar{a}_{2},\bar{k}_{2})$ is also called a flag datum of $\crr$ for convenience. Sequences of $0$ will be denoted by $0_{n}$ for simplicity, such as sequence $0,0,0,0,0$ will be denoted by $0_{5}$ and $(0,0),(0,0)$ will be denoted by $(0_{2})_{2}$, respectively.

Computing Eqs. (F1)-(F21) by Eq.~(\ref{formulas:flag and D}), we totally get $20$ different cases of all the flag datums, refer to the second column of Table~\ref{table:classifying flag datum}. Computing Eqs.(E1)-(E7) in Definition ~\ref{defn:datum equivalent} by Eq.~(\ref{formulas:flag and D}), we obtain the third and fourth columns of Table~\ref{table:classifying flag datum}, then we classify all the flag datums, the result is in the fifth column of Table~\ref{table:classifying flag datum}.
By Eqs.~(\ref{formulas:unified to extension}) (\ref{formulas:flag and extending structure}) (\ref{formulas:flag and D}) and for each flag datum $(\bar{l},\bar{r},\bar{t}_{r},\bar{t}_{l},
\bar{a}_{1},\bar{k}_{1},(\bar{b}_{e},\bar{b}_{\sigma}),(\bar{k}_{e},\bar{k}_{\sigma}),\bar{a}_{2},$ $\bar{k}_{2})$, suppose that $E={\bf k}\{e_{1},e_{2}\}$, define
\begin{align}\label{formulas:flag to extension}
\begin{aligned}
e_{1}\cdot_{E}e_{1}&=e_{1},\\
e_{2}\cdot_{E}e_{1}&=\bar{t}_{r}e_{1}+\bar{r}e_{2},\\
P_{e,\,E}(e_{1})&=0,\\
P_{\sigma,\,E}(e_{1})&=0,\\
\theta_{E}(e_{1})&=e_{1},\\
\end{aligned}\quad
\begin{aligned}
e_{1}\cdot_{E}e_{2}&=\bar{t}_{l}e_{1}+\bar{l}e_{2},\\
e_{2}\cdot_{E}e_{2}&=\bar{a}_{1}e_{1}+\bar{k}_{1}e_{2},\\
P_{e,\,E}(e_{2})&=\bar{b}_{e} e_{1}+\bar{k}_{e} e_{2},\\
P_{\sigma,\,E}(e_{2})&=\bar{b}_{\sigma}e_{1}+\bar{k}_{\sigma}e_{2},\\
\theta_{E}(e_{2})&=\bar{a}_{2}e_{1}+\bar{k}_{2}e_{2},
\end{aligned}
\end{align}
then $\ee$ is an extension of $R$, and the map $\psi$
 in Diagram ~(\ref{diagram:equivalent cohomologous ee}) is obtained by \[\psi(e_{1})=e_{1},\,\psi(e_{2})=\bar{g}e_{1}+\bar{h}e_{2}.\]
\end{exam}

%

\subsection{Matched pairs of Rota-Baxter family Hom-associative algebras}
In this subsection, we consider the factorization problem for Rota-Baxter family Hom-associative algebras, which is a subproblem of the ES problem.

\begin{defn}
Let $\crr$ and $\cvv$ be two Rota-Baxter family Hom-associative algebras. Suppose that
\begin{align*}
\triangleright:\,R\times V\rightarrow V,\quad\triangleleft:\,V\times R\rightarrow V,\quad \rightharpoonup:\,V\times R\rightarrow R,\quad\leftharpoonup:\,R\times V\rightarrow R
\end{align*}
are bilinear maps. The system $(\crr,\cvv,\triangleright,\triangleleft,\rightharpoonup,\leftharpoonup)$ is called a {\bf matched pair} of Rota-Baxter family Hom-associative algebras if the following conditions hold for all $a,b\in R$ and $x,y\in V$:
\begin{align*}
        &(M1)\,\,\big(V,(P_{\omega,\,V})_{\omega\in S},\theta_{V},\triangleright,\triangleleft\big)\,\,is\,\,a\,\, \crr\,\, bimodule;\\
        &(M2)\,\,\big(R,(P_{\omega})_{\omega\in S},\theta,\rightharpoonup,\leftharpoonup\big)\,\,is\,\,a\,\,\cvv\,\, bimodule;\\
        &(M3)\,\,(ab)\leftharpoonup \theta_{V}(x)=\theta(a)\,(b\leftharpoonup x)+\theta(a)\leftharpoonup(b\triangleright x);\\
        &(M4)\,\,(x\rightharpoonup a)\, \theta(b)+(x\triangleleft a)\rightharpoonup \theta(b)=\theta_{V}(x)\rightharpoonup(ab);\\
        &(M5)\,\,(a\leftharpoonup x)\, \theta(b)+(a\triangleright x)\rightharpoonup \theta(b)=\theta(a)\,(x\rightharpoonup b)+\theta(a)\leftharpoonup(x\triangleleft b);\\
        &(M6)\,\,(x\cdot_{V} y)\triangleleft\theta(a)=\theta_{V}(x)\triangleleft(y\rightharpoonup a)+\theta_{V}(x)\cdot_{V}(y\triangleleft a);\\
        &(M7)\,\,(a\leftharpoonup x)\triangleright\theta_{V}(y)+(a\triangleright x)\cdot_{V} \theta_{V}(y)= \theta(a)\triangleright (x\cdot_{V} y) ;\\
        &(M8)\,\,(x\rightharpoonup a)\triangleright \theta_{V}(y)+(x\triangleleft a)\cdot_{V} \theta_{V}(y)=\theta_{V}(x)\triangleleft(a\leftharpoonup y)+ \theta_{V}(x)\cdot_{V}(a\triangleright y).
    \end{align*}
\end{defn}

\begin{prop}\label{prop:matched pair unified product}
The extending datum $\Omega(R,V)=\big(\triangleright,\triangleleft,\rightharpoonup,\leftharpoonup,\cdot_{V},(P_{\omega,\,V})_{\omega\in S},\theta_{V}\big)$ is a Rota-Baxter family Hom extending structure of $\crr$ through $V$, if and only if, the system $(\crr,\cvv,\triangleright,\triangleleft,\rightharpoonup,\leftharpoonup)$ is a matched pair of Rota-Baxter family Hom-associative algebras. In the case, the associated unified product $\rnv$ is called the {\bf bicrossed product}, and it is denoted by $\crr\bowtie \cvv$.
\end{prop}

\begin{proof}
It's obtained directly by Theorem ~\ref{thm:datum and unified product} if the maps $f,(Q_{\omega})_{\omega\in S}$ and $\eta$ are all trivial.
\end{proof}


Applying the bijection $\Upsilon_{1}$ (or $\Upsilon_{2}$) in Proposition~\ref{prop:bijection of extension and unified product}, we have the following proposition.
\begin{prop}\label{prop:factor through matched pair}
Let $\crr$ and $\cvv$ be two Rota-Baxter family Hom-associative algebra, vector space $E=R+V$, $R\cap V=\{0\}$, and $\rho:E\rightarrow R$ the retraction associated to V. Suppose that $R\subset E $ is an extension of Rota-Baxter family Hom-associative algebras, then the following conditions are equivalent:
\begin{enumerate}
\item\label{prop:factor through matched pair1} The Rota-Baxter family Hom-associative algebra $E$ factorizes through $R$ and $V$.
\item\label{prop:factor through matched pair2} The unified product corresponding to the Rota-Baxter family Hom extending structure $\Omega(R,V)=\Upsilon_{1}(E)$ is a bicrossed product $R\bowtie V$.
\end{enumerate}
\end{prop}

\begin{proof}

First, we prove that ~\ref{prop:factor through matched pair1} $\Rightarrow$ ~\ref{prop:factor through matched pair2}. Since $V$ is a subalgebra of $E$, for all $x,y\in V$ and $\omega\in S$ and by Eq. ~(\ref{formulas:extension to unified}), we have
\begin{align*}
f(x,y)&=\rho(x\cdot_{E} y)=\rho(x\cdot_{V} y)=0,\text{ i.e., $f$ is trivial};\\
Q_{\omega}(x)&=\rho\big(P_{\omega,\,E}(x)\big)=\rho\big(P_{\omega,\,V}(x)\big)=0,\text{ i.e., $Q_{\omega}$ is trivial};\\
\eta(x)&=\rho\big(\theta_{E}(x)\big)=\rho\big(\theta_{V}(x)\big)=0,\text{ i.e., $\eta$ is trivial}.
\end{align*}
According to Proposition ~\ref{prop:matched pair unified product}, the corresponding unified product is a bicrossed product $R\bowtie V$.

Second, we prove that ~\ref{prop:factor through matched pair2} $\Rightarrow$ ~\ref{prop:factor through matched pair1}. Here we just need to prove that $V$ is a subalgebra of $E$. In fact, for all $x,y\in V$ and $\omega\in S$, by Eq. ~(\ref{formulas:unified to extension}), we have
\begin{align*}
    x\cdot_{E} y&=x\cdot_{V} y;\\
    P_{\omega,\,E}(x)&=P_{\omega,\,V}(x);\\
    \theta_{E}(x)&=\theta_{V}(x),
    \end{align*}
so $V$ is a subalgebra of $E$.
\end{proof}
\begin{exam} \label{exam:matched pair}
To continue Example~\ref{exam:ES problem}, define $e_{2}\cdot_{V} e_{2}=e_{2}$, $\theta_{V}(e_{2})=e_{2}$, $P_{e,V}(e_{2})=P_{\sigma,V}(e_{2})=-\lambda e_{2}$ on the vector space $V={\bf k}\{e_{2}\}$. Then $(V,\cdot_{V},(P_{e,V},P_{\sigma,V}),\theta_{V})$ is a Rota-Baxter family Hom-associative algebra by Example~\ref{exam:dendriform algebra}~\ref{exam:dendriform algebra2}.
By Table ~\ref{table:classifying flag datum} in Example~\ref{exam:ES problem}, we obtain all the matched pairs of $\crr$ and $V$ as follows: by Proposition ~\ref{prop:matched pair unified product} and Eqs. ~(\ref{formulas:flag and extending structure})-(\ref{formulas:flag and D}), collecting all the flag datums $(\bar{l},\bar{r},\bar{t}_{r},\bar{t}_{l},
\bar{a}_{1},\bar{k}_{1},(\bar{b}_{e},\bar{b}_{\sigma}),(\bar{k}_{e},\bar{k}_{\sigma}),\bar{a}_{2},$ $\bar{k}_{2})$ such that $\bar{a}_{1}=\bar{b}_{e}=\bar{b}_{\sigma}=\bar{a}_{2}=0$, $\bar{k}_{e}=\bar{k}_{\sigma}=-\lambda$ and $\bar{k}_{1}=\bar{k}_{2}=1$, then we obtain Table ~\ref{table:matched pairs} and each row in Table ~\ref{table:matched pairs} defines a matched pair $(\crr,\cvv,\triangleright,\triangleleft,\rightharpoonup,\leftharpoonup)$.
By Eq.~(\ref{formulas:flag to extension}) and Proposition~\ref{prop:factor through matched pair}, for each row, we obtain an extension $\ee$
that factorizes through $\crr$ and $V$.
\begin{table}[htbp]
  \centering
  \renewcommand\arraystretch{1.3}
\begin{tabular}{cccccccccccc}
  \toprule
  $\bar{l}$ & $\bar{r}$ & $\bar{t}_{r}$ & $\bar{t}_{l}$ & $\bar{a}_{1}$ & $\bar{k}_{1}$ & $\bar{b}_{e}$ & $\bar{b}_{\sigma}$ & $\bar{k}_{e}$ & $\bar{k}_{\sigma}$ & $\bar{a}_{2}$ & $\bar{k}_{2}$ \\
\midrule
$0$ & $0$ & $0$ & $0$ & \multirow{5}*{0} & \multirow{5}*{1} & \multirow{5}*{0} & \multirow{5}*{0} &\multirow{5}*{$-\lambda$}&\multirow{5}*{$-\lambda$}&\multirow{5}*{0}&\multirow{5}*{1}\\
$0$ & $0$ & $1$ & $1$ &  &  &  & &&&&\\
$1$ & $0$ & $1$ & $0$ &  &  &  & &&&&\\
$0$ & $1$ & $0$ & $1$ &  &  &  & &&&&\\
$1$ & $1$ & $0$ & $0$ &  &  &  & &&&&\\
\bottomrule
\end{tabular}
\\[5pt]
  \caption{matched pairs of $\crr$ and $V$}\label{table:matched pairs}
\end{table}
\end{exam}
\section{Classifying complements for Rota-Baxter family Hom-associative algebras}
In this section, we define deformation maps on a Rota-Baxter family Hom extending structure and deformations of $V$, then we theoretically solve the CCP problem and give an example.

%
%
%
%

\begin{defn}
Let $\Omega(\crr,\cxx)$ be a Rota-Baxter family Hom extending structure of $\crr$ through $V$. A linear map $d:V\rightarrow R$ is called a {\bf deformation map} of $\Omega(\crr,\cxx)$, if the following conditions hold for all $x,y\in V$:
\allowdisplaybreaks{
\begin{align}
d(x)\,d(y)-d(x\cdot_{V} y)&=d(d(x)\triangleright y+x \triangleleft d(y))-d(x)\leftharpoonup y-x\rightharpoonup d(y)-f(x,y);\label{formulas:deformation in defn1}\\
d(P_{\omega,\,V}(x))&=Q_{\omega}(x)+P_{\omega}(d(x));\label{formulas:deformation in defn2}\\
d(\theta_{V}(x))&=\eta(x)+\theta(d(x)).\label{formulas:deformation in defn3}
\end{align}
}
\end{defn}

We denote by $\mathcal{D}(\cxx,\crr)$ the set of all deformation maps $d$ of the Rota-Baxter family Hom extending structure  $\Omega(\crr,\cxx)$ of $\crr$ through $V$.
\begin{prop}\label{prop:demormation of X}
Let $\Omega(\crr,\cxx)$ be a Rota-Baxter family Hom extending structure of $\crr$ through $V$ and $d:V\rightarrow R$ a deformation map of $\Omega(\crr,\cxx)$. Define a new multiplication on the vector space $V$:
\begin{align}\label{formulas:deformation1}
x\cdot_{d}y:=x\cdot_{V} y+d(x)\triangleright y+x \triangleleft d(y),\text{ for all }x,y\in V.
\end{align}
Then $\cxx_{d}:=\big(V,\cdot_{d},(P_{\omega,\,V})_{\omega\in S},\theta_{V}\big)$ is a Rota-Baxter family Hom-associative algebra, and $\cxx_{d}$ is called the {\bf deformation} of $\cxx$.
\end{prop}

\begin{proof}%
First, we give some useful identities.
In Diagram ~(\ref{diagram:equivalent cohomologous ue}), define the canonical projection $\pi_{\crr}:\rnv\rightarrow \crr$ and the injection $i_{V}:V\rightarrow \rnv$ as follows:
$$\pi_{\crr}(a,x)=a,\quad i_{V}(x)=(0,x),\,\, a\in \crr,\,x\in V.$$
Then for all $x,y\in V$, we have
\begin{align}
x\cdot_{d}y=&x\cdot_{V}y+d(x)\triangleright y+x \triangleleft d(y)\nonumber\\
           =&\pi^{\prime}((d(x),x)\,\bar{\cdot}\, (d(y),y));\label{zdeformation in proof1}\quad\text{(by Eq.~\eqref{formulas:unified product})}\\
d(x\cdot_{d}y)=&d(x\cdot_{V}y+d(x)\triangleright y+x \triangleleft d(y))\nonumber\\
              =&d(x)\, d(y)+x\rightharpoonup d(y)+d(x)\leftharpoonup y+f(x,y)\quad\text{(by Eq.~\eqref{formulas:deformation in defn1})}\nonumber\\
              =&\pi_{\crr}((d(x),x)\,\bar{\cdot}\, (d(y),y));\label{zdeformation in proof2}\quad\text{(by Eq.~\eqref{formulas:unified product})}\\
(d(x\cdot_{d}y),x\cdot_{d}y)=&(d(x),x)\,\bar{\cdot}\, (d(y),y);\quad\text{(by Eqs.~\eqref{zdeformation in proof1}-\eqref{zdeformation in proof2})}\label{zdeformation in proof3}\\
\big(d(\theta_{V}(z)),\theta_{V}(z)\big)=&\big(\eta(z)+\theta(d(z)),\theta_{V}(z)\big)\quad\text{(by Eq.~\eqref{formulas:deformation in defn3})}\nonumber\\
=&\bar{\theta}\big(d(z),z\big);\quad\text{(by Eq.~\eqref{formulas:unified product})}\label{zdeformation in proof4}\\
(d\big(P_{\beta,\,V}(y)\big),P_{\beta,\,V}(y))=&(Q_{\beta}(y)+P_{\beta}(d(y)),P_{\beta,\,V}(y))
\quad\text{(by Eq.~\eqref{formulas:deformation in defn2})}\nonumber\\
=&\bar{P}_{\beta}(d(y),y).\quad\text{(by Eq.~\eqref{formulas:unified product})}\label{zdeformation in proof5}
\end{align}

Second, we prove that $\cxx_{d}:=\big(V,\cdot_{d},(P_{\omega,\,V})_{\omega\in S},\theta_{V}\big)$ is a Rota-Baxter family Hom-associative algebra.
For all $x,y,z\in V$, we have
\begin{align*}
(x\cdot_{d}y)\cdot_{d}\theta_{V}(z)
=&\pi^{\prime}\Big(\big(d(x\cdot_{d}y),x\cdot_{d}y\big)\,\bar{\cdot}\,\big(d(\theta_{V}(z)),\theta_{V}(z)\big)\Big)\quad\text{(by Eq.~\eqref{zdeformation in proof1}})\\
=&\pi^{\prime}\Big(((d(x),x)\,\bar{\cdot}\, (d(y),y)\big)\,\bar{\cdot}\,\bar{\theta}\big(d(z),z\big)\Big)\quad\text{(by Eqs.~\eqref{zdeformation in proof3}-\eqref{zdeformation in proof4})}\\
=&\pi^{\prime}\Big(\bar{\theta}\big(d(x),x\big)\,\bar{\cdot}\, \big((d(y),y)\,\bar{\cdot}\,(d(z),z)\big)\Big)\\
&\quad\text{(by $\crr \natural V$ being a Rota-Baxter family Hom-associative algebra)}\\
=&\pi^{\prime}\Big((d(\theta_{V}(x)),\theta_{V}(x))\,\bar{\cdot}\, \big(d\big(y\cdot_{d}z\big),y\cdot_{d}z\big)\Big)\quad\text{(by Eqs.~\eqref{zdeformation in proof3}-\eqref{zdeformation in proof4})}\\
=&\theta_{V}(x)\cdot_{d}(y\cdot_{d}z);\quad\text{(by Eq.~\eqref{zdeformation in proof1})}\\
P_{\alpha,\,V}(x)\cdot_{d}P_{\beta,\,V}(y)
=&\pi^{\prime}\Big((d\big(P_{\alpha,\,V}(x)\big),P_{\alpha,\,V}(x))\,\bar{\cdot}\,
(d\big(P_{\beta,\,V}(y)\big),P_{\beta,\,V}(y))\Big)\quad\text{(by Eq.~\eqref{zdeformation in proof1})}\\
=&\pi^{\prime}\Big(\bar{P}_{\alpha}(d(x),x)\,\bar{\cdot}\,\bar{P}_{\beta}(d(y),y)\Big)\quad\text{(by Eq.~\eqref{zdeformation in proof5})}\\
=&\pi^{\prime}\Big(\bar{P}_{\alpha\beta}\Big(\bar{P}_{\alpha}(d(x),x)\,\bar{\cdot}\,
(d(y),y)+(d(x),x)\,\bar{\cdot}\,\bar{P}_{\beta}(d(y),y)
+\lambda\,(d(x),x)\,\bar{\cdot}\,(d(y),y)\Big)\Big)\\
&\quad\text{(by $\crr \natural V$ being a Rota-Baxter family Hom-associative algebra)}\\
=&P_{\alpha\beta,\,V}\Big(\pi^{\prime}\Big(\bar{P}_{\alpha}(d(x),x)\,\bar{\cdot}\,
(d(y),y)+(d(x),x)\,\bar{\cdot}\,\bar{P}_{\beta}(d(y),y)
+\lambda\,(d(x),x)\,\bar{\cdot}\,(d(y),y)\Big)\Big)\\
&\hspace{3cm}\text{(by Eq.~\eqref{formulas:unified product})}\\
=&P_{\alpha\beta,\,V}\Big(\pi^{\prime}\Big((d\big(P_{\alpha,\,V}(x)\big),P_{\alpha,\,V}(x))\,\bar{\cdot}\,
(d(y),y)+(d(x),x)\,\bar{\cdot}\,(d\big(P_{\beta,\,V}(y)\big),P_{\beta,\,V}(y))\\
&+\lambda\,(d(x),x)\,\bar{\cdot}\,(d(y),y)\Big)\Big)\quad\text{(by Eq.~\eqref{zdeformation in proof5})}\\
=&P_{\alpha\beta,\,V}\Big(\pi^{\prime}\Big((d\big(P_{\alpha,\,V}(x)\big),P_{\alpha,\,V}(x))\,\bar{\cdot}\,
(d(y),y)\Big)+\pi^{\prime}\Big((d(x),x)\,\bar{\cdot}\,(d\big(P_{\beta,\,V}(y)\big),P_{\beta,\,V}(y))\Big)
\\
&+\pi^{\prime}\Big(\lambda\,(d(x),x)\,\bar{\cdot}\,(d(y),y)\Big)\Big)\\
=&P_{\alpha\beta,\,V}(P_{\alpha,\,V}(x)\cdot_{d} y+x\cdot_{d}P_{\beta,\,V}(y)+\lambda\,x\cdot_{d} y)\quad\text{(by Eq.~\eqref{zdeformation in proof1})}.
\end{align*}
By Eq. ~\eqref{formulas:left module definition3} in (R1) of Rota-Baxter family Hom extending structure $\Omega(R,V)$, we obtain that $ V_{d}$ is a Rota-Baxter family Hom-associative algebra.
\end{proof}

%

\begin{thm}\label{thm:complement to deformation}
Let $\crr\subset E$ be an extension of Rota-Baxter family Hom-associative algebras with a linear retraction $\rho:\cee\rightarrow \crr$, $V:=\mathrm{ker}\,\rho$ and the Rota-Baxter family Hom extending structure $\Omega(\crr,\cxx)=\Upsilon_{1}(E)$ by {\rm Remark ~\ref{rem:map unified to extension} \ref{rem:map unified to extension2}}. Suppose that $B$ is a Rota-Baxter family Hom complement of $\crr$ in $E$, then there exists a deformation map $d:V\rightarrow R$ of $\Omega(\crr,\cxx)$, such that $B\cong \cxx_{d}$ as Rota-Baxter family Hom-associative algebras.
\end{thm}

\begin{proof}
Let $\tilde{d}:E\rightarrow \crr$ be the linear retraction associated to $B$,
and $d:=-\tilde{d}|_{V}$.

First, we prove that $d$ is a deformation map.
In fact, $\tilde{d}(x-\tilde{d}(x))=\tilde{d}(x)-\tilde{d}(x)=0$, for all $x\in V$, then we have $x-\tilde{d}(x)=x+d(x)\in \mathrm{ker}\,\tilde{d}=B$ by {\rm Remark ~\ref{rem:map unified to extension} \ref{rem:map unified to extension1}}. Moreover, $B$ is a Rota-Baxter family Hom complement of $\crr$ in $E$, then for all $x,y\in V$, we have
\begin{align*}
0=&\tilde{d}\Big(\big(x+d(x)\big)\cdot_{E}\big(y+d(y)\big)\Big)\\
=&\tilde{d}\big(d(x)\,d(y)+x\rightharpoonup d(y)+d(x)\leftharpoonup y+f(x,y)+x\triangleleft d(y)+d(x)\triangleright y+x\cdot_{V} y\big)\quad\text{(by Eq.~\eqref{formulas:unified to extension})}\\
=&d(x)\,d(y)+x\rightharpoonup d(y)+d(x)\leftharpoonup y+f(x,y)-d(x\triangleleft d(y)+d(x)\triangleright y)-d(x\cdot_{V} y);\\
0=&\tilde{d}(P_{\omega,\,E}(x+d(x)))\\
=&\tilde{d}(P_{\omega}(d(x))+Q_{\omega}(x)+P_{\omega,\,V}(x))\quad\text{(by Eq.~\eqref{formulas:unified to extension})}\\
=&P_{\omega}(d(x))+Q_{\omega}(x)-d(P_{\omega,\,V}(x));\\
0=&\tilde{d}(\theta_{E}(x+d(x)))\\
=&\tilde{d}(\theta(d(x))+\eta(x)+\theta_{V}(x))\quad\text{(by Eq.~\eqref{formulas:unified to extension})}\\
=&\theta(d(x))+\eta(x)-d(\theta_{V}(x)).
\end{align*}
We have proved Eqs.~\eqref{formulas:deformation in defn1}-\eqref{formulas:deformation in defn3}.

Second, define a linear map
\begin{align*}
\phi:V_{d}&\rightarrow B\\
x&\mapsto x+d(x),\text{ $x\in V$.}
\end{align*}
Similar to ~\cite[Theorem 5.3]{ZW24},
here we only prove that the linear map $\phi$ is a morphism of Rota-Baxter family Hom-associative algebras. In fact, for all $x,y\in V$ and $\omega\in S$, we have
\begin{align*}
\phi(x\cdot_{d}y)=&x\cdot_{d}y+d(x\cdot_{d}y)\\
=&x\triangleleft d(y)+d(x)\triangleright y+x\cdot_{V} y+d(x\triangleleft d(y)+d(x)\triangleright y+x\cdot_{V} y)\quad\text{(by Eq.~\eqref{formulas:deformation1})}\\
=&x\triangleleft d(y)+d(x)\triangleright y+x\cdot_{V} y+d(x\triangleleft d(y)+d(x)\triangleright y)+d(x\cdot_{V} y)\\
=&x\triangleleft d(y)+d(x)\triangleright y+x\cdot_{V} y+d(x)\,d(y)+x\rightharpoonup d(y)+d(x)\leftharpoonup y+f(x,y)\\
&\hspace{3cm}\text{(by Eq.~\eqref{formulas:deformation in defn1})}\\
=&\big(x+d(x)\big)\cdot_{E}\big(y+d(y)\big)\quad\text{(by Eq.~\eqref{formulas:unified to extension})}\\
=&\phi(x)\cdot_{E}\phi(y).\\
\phi(P_{\omega,\,V}(x))
=&P_{\omega,\,V}(x)+d(P_{\omega,\,V}(x))\\
=&P_{\omega,\,V}(x)+Q_{\omega}(x)+P_{\omega}(d(x))\quad\text{(by Eq.~\eqref{formulas:deformation in defn2})}\\
=&P_{\omega,\,E}(x)+P_{\omega,\,E}(d(x))\quad\text{(by Eq.~\eqref{formulas:unified to extension})}\\
=&P_{\omega,\,E}(x+d(x))\\
=&P_{\omega,\,E}(\phi(x));\\
%
\phi(\theta_{V}(x))
=&\theta_{V}(x)+d(\theta_{V}(x))\\
=&\theta_{V}(x)+\eta(x)+\theta(d(x))\quad\text{(by Eq.~\eqref{formulas:deformation in defn3})}\\
=&\theta_{E}(x)+\theta_{E}(d(x))\quad\text{(by Eq.~\eqref{formulas:unified to extension})}\\
=&\theta_{E}(x+d(x))\\
=&\theta_{E}(\phi(x)).
\end{align*}

\end{proof}

\begin{rem}\label{rem:complement to deformation}
Under the condition of Theorem ~\ref{thm:complement to deformation}, we denote by $\mathcal{C}(\crr,E)$ the set of all Rota-Baxter family Hom complements of $\crr$ in $E$, then we establishes a map:
\begin{align*}
\Delta_{1}:\mathcal{C}(\crr,E)&\rightarrow \mathcal{D}(\crr,V)\\
B&\mapsto d:=-\tilde{d}|_{V},
\end{align*}
where $\tilde{d}$ is the retraction associated to $B$.
\end{rem}

\begin{prop}\label{prop:bijection complement to deformation}
The map $\Delta_{1}:\mathcal{C}(\crr,E)\rightarrow \mathcal{D}(\crr,V)$ defined in {\rm Remark ~\ref{rem:complement to deformation}} is a bijection.
\end{prop}

\begin{proof}

Defining a map:
\begin{align*}
\Delta_{2}:\mathcal{D}(\crr,V)&\rightarrow \mathcal{C}(\crr,E)\\
d&\mapsto B:=\mathrm{ker}\, \xi,
\end{align*}
where the linear map $\xi:E\rightarrow \crr$ is defined by $\xi(a+x)=a-d(x)$,\,$a\in \crr$, $x\in V$. We need to prove that the map $\Delta_{2}$ is well defined, i.e., $B$ is a  Rota-Baxter family Hom complement of $\crr$ in $E$. Suppose that the inclusion map $i:\crr\rightarrow E$ is defined by $i(a)=a,\,a\in \crr$, then $\xi(i(a))=\xi(a)=a$, i.e., the map $\xi$ is a linear retraction. By Remark ~\ref{rem:map unified to extension} \ref{rem:map unified to extension1}, $B$ is a space complement of $\crr$ in $E$. We only need to prove that the operations $\cdot_{E},P_{\omega,\,E}$ and $\theta_{E}$ is closed on $B$. In fact,
\begin{align*}
B=&\mathrm{ker}\, \xi\\
 =&\{a+x|\xi(a+x)=a-d(x)=0,\,a\in \crr, x\in V\}\\
  =&\{a+x|a=d(x),\,a\in \crr, x\in V\}\\
=&\{d(x)+x|x\in V\}.
\end{align*}
Also we have
\begin{align*}
\big(x+d(x)\big)\cdot_{E}\big(y+d(y)\big)
=&d(x)\, d(y)+x\rightharpoonup d(y)+d(x)\leftharpoonup y+f(x,y)+x\cdot_{V}y+d(x)\triangleright y+x \triangleleft d(y)\\
&\hspace{3cm}\text{(by Eq.~\eqref{formulas:unified to extension})}\\
=&d(d(x)\triangleright y+x \triangleleft d(y))+d(x\cdot_{V}y)+x\cdot_{V}y+d(x)\triangleright y+x \triangleleft d(y)\\
&\hspace{3cm}\text{(by Eq.~\eqref{formulas:deformation in defn1})}\\
=&d(x\cdot_{V}y+d(x)\triangleright y+x \triangleleft d(y))+x\cdot_{V}y+d(x)\triangleright y+x \triangleleft d(y)\in B;\\
%
P_{\omega,\,E}(x+d(x))=&P_{\omega,\,V}(x)+Q_{\omega}(x)+P_{\omega}(d(x))\quad\text{(by Eq.~\eqref{formulas:unified to extension})}\\
=&P_{\omega,\,V}(x)+d(P_{\omega,\,V}(x))\quad\text{(by Eq.~\eqref{formulas:deformation in defn2})}\\
=&P_{\omega,\,V}(x)+d(P_{\omega,\,V}(x))\in B;\\
%
\theta_{E}(x+d(x))=&\theta_{V}(x)+\eta(x)+\theta(d(x))\quad\text{(by Eq.~\eqref{formulas:unified to extension})}\\
=&\theta_{V}(x)+d(\theta_{V}(x))\quad\text{(by Eq.~\eqref{formulas:deformation in defn3})}\\
=&\theta_{V}(x)+d(\theta_{V}(x))\in B.
\end{align*}
Similar to ~\cite[Proposition 5.5]{ZW24}, we can prove that 
$\Delta_{1}\circ\Delta_{2}=\mathrm{id}_{\mathcal{D}(\crr,V)}$ and $\Delta_{2}\circ\Delta_{1}=\mathrm{id}_{\mathcal{C}(\crr,E)}$.
%
%
%
\end{proof}

\begin{defn}\label{definition:equivalent deformation}
Let $\Omega(\crr,\cxx)$ be a Rota-Baxter family Hom extending structure of $\crr$ through $V$. Two deformation maps $d,D:V\rightarrow R$ are called {\bf equivalent}, and we denote it by $d\sim D$, if there exists a linear automorphism $\delta:V\rightarrow V$ satisfying the following conditions for all $x,y\in V$ and $\omega\in S$,
\begin{align}
\delta(x\cdot_{V} y)-\delta(x)\cdot_{V}\delta(y)&=D(\delta(x))\triangleright \delta(y)+\delta(x)\triangleleft D(\delta(y))-\delta(d(x)\triangleright y)-\delta(x \triangleleft d(y));\label{formula:equivalent deformation1}\\
\delta(P_{\omega,\,V}(x))&=P_{\omega,\,V}(\delta(x));\label{formula:equivalent deformation2}\\
\delta(\theta_{V}(x))&=\theta_{V}(\delta(x)).\label{formula:equivalent deformation3}
\end{align}
\end{defn}
By Definition ~\ref{definition:equivalent deformation}, we know that $d\sim D$ if and only if $\cxx_{d}\cong \cxx_{D}$ as Rota-Baxter family Hom-associative algebras. It's obvious that the relation $\sim$ is an equivalence relation. Suppose that
$$CH^{2}(\crr,E):=\mathcal{C}(\crr,E)/\cong,\quad H^{2}(\crr,V):=\mathcal{D}(\crr,V)/\sim.$$
Now we arrive at our main result of classifying complements for Rota-Baxter family Hom-associative algebras.
\begin{thm}
Let $\crr\subset E$ be an extension of Rota-Baxter family Hom-associative algebras with a linear retraction $\rho:\cee\rightarrow \crr$, $V=\mathrm{ker}\,\rho$ and the Rota-Baxter family Hom extending structure $\Omega(\crr,V)=\Upsilon_{1}(E)$ by {\rm Remark~\ref{rem:map unified to extension}~\ref{rem:map unified to extension2}}.
Then the map $\Delta_{1}$ defined in {\rm Remark ~\ref{rem:complement to deformation}} induces a bijection of equivalence classes via $\sim$:
%
\begin{align*}
\Delta:CH^{2}(\crr,E)&\rightarrow H^{2}(\crr,V)\\
     [B]\,\,\,\,\,\,\,\,\,\,\,\,&\mapsto \,\,\,\,\,\,\,\,\,\,\,\,[d]
\end{align*}
In particular, the index of $\crr$ in $E$ is computed by the formula $[E:\crr]=|H^{2}(\crr,V)|$.
\end{thm}

\begin{proof}
It's a direct result of Proposition ~\ref{prop:induce map}.
\end{proof}
\begin{exam}
To continue Example ~\ref{exam:ES problem}, suppose that $d(e_{2})=\bar{d}e_{1}\,(\bar{d} \in {\bf k})$, $\delta(e_{2})=\bar{\delta}e_{2}\,(\bar{\delta} \neq 0\in {\bf k})$ and $x=y=e_{2}$. Selecting case 10 in Table ~\ref{table:classifying flag datum}, for each flag datum of $\crr$ (i.e., for different $\bar{k}_{2}$ and $\bar{t}_{r}$), we obtain an extension $\ee$ of $\crr$ by Eq. ~\eqref{formulas:flag to extension}.
\begin{enumerate}
\item When $\bar{k}_{2}\neq 1$ and given $\bar{t}_{r}$, computing Eqs. ~\eqref{formulas:deformation in defn1}-\eqref{formulas:deformation in defn3} by Eqs.~\eqref{formulas:flag and extending structure}-\eqref{formulas:flag and D}, we obtain the deformation map $\bar{d}=-\bar{t}_{r}$.
By Proposition ~\ref{prop:bijection complement to deformation}, we obtain a Rota-Baxter family Hom complement $B$ of $\crr$ in $\cee$ as follows:
 $B=\mathrm{ker}\, \xi=\{d(x)+x|x\in V\}={\bf k}\{e_{2}-\bar{t}_{r}e_{1}\}$.
 In particular, $[E:R]=1.$
\item When $\bar{k}_{2}=1$ and given $\bar{t}_{r}$, we obtain infinity many deformation maps $\bar{d}\in {\bf k}$. For each deformation map $\bar{d} \neq -\bar{t}_{r}$, computing Eqs.~\eqref{formula:equivalent deformation1}-\eqref{formula:equivalent deformation3} by Eqs.~\eqref{formulas:flag and extending structure}-\eqref{formulas:flag and D}, we obtain $\bar{\delta}=\bar{d}+\bar{t}_{r}$, then we know that it is equivalent to the deformation map $1-\bar{t}_{r}$. In particular $[E: R]=2.$
\end{enumerate}
\end{exam}

\noindent
{\bf Declaration of interests.}
The authors have no conflicts of interest to disclose.

\smallskip
\noindent
{{\bf Acknowledgments.}
Y. Y. Zhang is supported by National Natural Science Foundation of China (12101183) and also supported by the Postdoctoral Fellowship Program of CPSF under
Grant Number (GZC20240406). }

\smallskip

\noindent
{\bf Data availability.}
Data will be made available on request.

\end{document}